\documentclass[11pt]{amsart}
\usepackage{verbatim}
\usepackage{amsmath,amssymb,amsfonts,bm,amscd,amsthm}
\usepackage{fancyhdr}
\usepackage{psfrag}
\usepackage{tikz}
\usepackage{hyperref}
\hypersetup{%
  colorlinks = true,
  linkcolor  = red
}
\usetikzlibrary{matrix}
\makeatletter
\@namedef{subjclassname@1991}{2020 Mathematics Subject Classification}
\makeatother

\numberwithin{equation}{section}

\newtheorem{thm}{Theorem}[section]
\newtheorem{corollary}[thm]{Corollary}
\newtheorem{lemma}[thm]{Lemma}
\newtheorem{proposition}[thm]{Proposition}

\newtheorem{example}[thm]{Example}

\newtheorem*{notation}{Notation}

\newtheorem{defn}[thm]{Definition}

\DeclareMathOperator{\gldim}{gl.dim}
\DeclareMathOperator{\Hom}{Hom}
\DeclareMathOperator{\End}{End}
\DeclareMathOperator{\Pd}{pd}
\DeclareMathOperator{\lead}{lead}

\renewcommand{\Im}{\textup{Im}}

\newtheorem*{Problem One}{Problem One}
\newtheorem*{Problem Two}{Problem Two}
\newtheorem*{Problem Three}{Problem Three}
\newtheorem*{Problems Four and Five}{Problems Four and Five}
\newtheorem*{Problem Five}{Problem Five}

\begin{document}
\title{{\bf CONSTRUCTING ENDOMORPHISM RINGS OF LARGE FINITE GLOBAL DIMENSION}}
\author{Ali Mousavidehshikh}
\address[Ali Mousavidehshikh]{Department of Mathematics and Computer Science\\University of Toronto, Mississauga Campus\\3359 Mississauga Road, Mississauga, ON, L5L 1C6, Canada}
\email{ali.mousavidehshikh@utoronto.ca}
\subjclass{18G05(primary),18G55(secondary)}
\date{}
\maketitle
\begin{center}$Dedicated~to~the~memory~of~Ragnar$-$Olaf~Buchweitz$\end{center}
\begin{abstract} In this paper we study endomorphism rings of finite global dimension over a ring associated to a numerical semigroup. We construct these endomorphism rings in two ways, called the lazy and greedy construction. The first main result of this paper shows that the lazy construction enables us to obtain endomorphism rings of arbitrarily large global dimension. The second main result of this paper shows that  the greedy construction gives us endomorphism rings which always have global dimension two. As a consequence, for a fixed numerical semigroup, the difference of the maximal possible value and the minimal possible value of the global dimension of an endomorphism ring over that ring can be arbitrarily large.\end{abstract}

\begin{center}\tableofcontents \end{center}

\section{Introduction}\label{Introduction}

\hspace{0.5cm} The global dimension of a ring is one of the most fundamental invariants. It measures the complexity of the category of modules over a ring $R$ by looking at how far $R$-modules are being from projective. It plays important roles in algebra and geometry. For example, Auslander-Buchsbaum-Serre Theorem characterizes commutative regular local rings in terms of finiteness of global dimension. 

\hspace{0.5cm} In representation theory, it often plays important roles to construct a finitely generated module $M$ over a given ring $R$ such that the endomorphism algebra $\End_{R}(M)$ has finite global dimension. A basic example appears in Auslander-Reiten theory: When $M$ is an additive generator of $_{R}\text{Mod}$ (finitely generated left $R$-modules, one can replace left by right), then $\End_{R}(M)$ has global dimension at most two (see \cite{A3} and \cite{ARS}). This gives a bijection $R\rightarrow \End_{R}(M)$ between representation-finite algebras and algebras with global dimension at most two and dominant dimension at least two. Another basic example due to Auslander shows that
\begin{eqnarray}\label{chain1}
\End_{R}(M),~\text{where}~M=\bigoplus_{i\geq 0}R/\text{rad}^{i}R,
\end{eqnarray}
has finite global dimension for any finite dimensional algebra $R$ (see \cite{A2,A1,ARS}).

\hspace{0.5cm} These classical results have been extensively studied by several authors, and a number of important applications are known, e.g. Auslander's representation dimension, Dlab-Ringel's approach to quasi-hereditary algebras of Cline-Parshall-Scott, Rouquier's dimensions of triangulated categories, cluster tilting in higher dimensional Auslander-Reiten theory, and non-commutative resolutions in algebraic geometry due to Van den Bergh and others. In Krull dimension one, there is a natural analog of the construction (\ref{chain1}). 
\begin{thm}\label{t1.3.1} Let 
\begin{eqnarray}\label{chain3}
(R,m)=(R_{1},m_{1})\subseteq (R_{2},m_{2})\subseteq \ldots \subseteq (R_{l-1},m_{l-1})\subseteq (R_{l},m_{l})
\end{eqnarray}
be a chain of local Noetherian rings, where for each $i$, $R_{i}$ is commutative, reduced, complete (with respect to its Jacobson radical), has Krull dimension one, and $R_{l}$ is regular. If $R_{i+1}\subseteq\End_{R_{1}}(m_{i})$ for $1\leq i\leq l-1$, then 
\begin{eqnarray}\label{chain2}
E:=\End_{R}(M),~\text{where}~M:=\bigoplus_{i=1}^{l}R_{i},
\end{eqnarray}
has global dimension at most $l$. 
\end{thm}
\begin{proof} See \cite{OI2} example 2.2.3 and \cite{GL}.\end{proof}

\hspace{0.5cm} The ring $R=R_{1}$ is called the starting ring for the chain (\ref{chain3}). In general, given a ring $R$ of Krull dimension one it is a hard problem to understand all the endomorphism rings $\End_{R}(M)$ with finite global dimension, since there are a huge number of modules $M$ with $\End_{R}(M)$ having finite global dimension. A more reasonable problem is to determine the set of all possible values of the global dimension of $\End_{R}(M)$ in (\ref{chain2}), which Ballard-Favero-Katzarkov call the global spectrum of $R$. If $R$ is a commutative, reduced, complete, local Noetherian ring with Krull dimension one, then its normalization is an endomorphism ring of finite global dimension, which has global dimension one (since it is regular). In particular, for such rings, one is always an element of the global spectrum of $R$.

\subsection{Conventions.}
A ring is said to be complete if it is complete with respect to its Jacobson radical.

\subsection{Structure of the paper.}
The structure of this paper is as follows: In section \ref{Numerical Semigroups and Numerical Semigroup Rings} we give some of the necessary background on numerical semigroups and introduce some of the notations and definitions which will be used throughout the paper. In section \ref{CRC} we define the notion of a radical chain and construct two such chains which we call the lazy and greedy construction. We also associate an endomorphism ring to each of these constructions. In section \ref{Projective section} we analyse the projective and simple modules over our endomorphism rings. In section \ref{The functor lceil rceil} we introduce the functor $\lceil ~\rceil$ and some of its properties. This functor plays a crucial role in the proofs of the main results in this paper. In section \ref{Family of Starting Rings} we prove the two main results of this paper, first of which gives us endomorphism rings with arbitrarily large (but finite) global dimension (Theorems \ref{t2.4.1}, \ref{t2.4.3}), and the second being the construction of endomorphism rings which always have global dimension two (Theorem \ref{t3.3.1}). 

\section{Numerical Semigroups and Numerical Semigroup Rings}\label{Numerical Semigroups and Numerical Semigroup Rings}

\hspace{0.5cm} Let $\mathbb{N}$ be the set of the positive integers and $\mathbb{N}_{0}$ be the set of the non-negative integers. A set $\mathcal{H}\subseteq \mathbb{N}_{0}$ is called a $numerical~semigroup$ if zero is an element of $\mathcal{H}$, it is closed under addition, and $\mathbb{N}_{0}\setminus\mathcal{H}$ is a finite set. The $Frobenius~number$ of $\mathcal{H}$, denoted by $F(\mathcal{H})$, is the largest integer not in $\mathcal{H}$ (this is a finite number as $\mathbb{N}_{0}\setminus\mathcal{H}$ is a finite set). Notice that $F(\mathcal{H})=-1$ if and only if $\mathcal{H}=\mathbb{N}_{0}$, otherwise $F(\mathcal{H})\geq 2$. We define $e(\mathcal{H})=\min\lbrace n\in \mathbb{N}:n\in \mathcal{H}\rbrace$, called the multiplicity of $\mathcal{H}$.

\hspace{0.5cm} Given $A=\lbrace \alpha_{1},\alpha_{2},\ldots,\alpha_{r}\rbrace\subseteq \mathbb{N}$, we say that $A$ generates a numerical semigroup $\mathcal{H}$ if 
\begin{eqnarray*}
\mathcal{H}=\langle A\rangle:=\lbrace x_{1}\alpha_{1}+x_{2}\alpha_{2}+\ldots +x_{s}\alpha_{r}:x_{i}\in \mathbb{N}_{0}\rbrace.
\end{eqnarray*} 
We call $A$ a $generating~set$ for $\mathcal{H}$. The set $A$ is called a $minimal~generating~set$ for $\mathcal{H}$ if no proper subset of $A$ is a generating set for $\mathcal{H}$. It is a standard fact that $\langle A\rangle$ forms a numerical semigroup if and only if $\gcd(A)=1$, and every numerical semigroup arises this way. Furthermore, every numerical semigroup has a unique minimal generating set, and this set has finitely many elements (see \cite{Rosales1} and \cite{Rosales2}).

\hspace{0.5cm} Let $k$ be a field. We define $R(\mathcal{H})$ to be the subring of $k[[t]]$ generated by $t^{n}$ over $k$ for all $n\in \mathcal{H}$. We call $R(\mathcal{H})$ the numerical semigroup ring associated to $\mathcal{H}$. More precisely, If $\lbrace \alpha_{1},\alpha_{2},\ldots,\alpha_{r}\rbrace$ is a minimal generating set for the numerical semigroup $\mathcal{H}$, then
\begin{eqnarray*}
R(\mathcal{H})=\left\lbrace \sum_{\substack{i\geq 0\\\text{finite}}}a_{i}t^{i}:~a_{i}\in k,~i\in \mathcal{H}\right\rbrace =k[[t^{\alpha_{1}},t^{\alpha_{2}},\ldots ,t^{\alpha_{r}}]].
\end{eqnarray*}
Notice that the normalization $\tilde{R}(\mathcal{H})$ of $R(\mathcal{H})$ is the ring of formal power series $k[[t]]$. We set $F(R(\mathcal{H}))=F(\mathcal{H})$ and $e(R(\mathcal{H}))=e(\mathcal{H})$. Given a ring $R(\mathcal{H})$, the principal ideal generated by $t^{a}$ in $R(\mathcal{H})$ is denoted by $t^{a}R(\mathcal{H})$.

\hspace{0.5cm} For any numerical semigroup $\mathcal{H}$, $R(\mathcal{H})$ is a local, commutative, Noetherian, reduced, complete ring that has Krull dimension 1. Moreover, the normalization of $R(\mathcal{H})$, denoted by $\tilde{R}(\mathcal{H})$, is $k[[t]]$ (which is a regular ring), and the total quotient ring of $R(\mathcal{H})$ (obtained by inverting all non-zero divisors in $R(\mathcal{H})$), denoted by $\overline{R}(\mathcal{H})$, is $k((t))$ (which is a field).

\begin{defn} \normalfont Suppose $\mathcal{H}$ is a numerical semigroup with minimal generating set $\lbrace \alpha_{1},\alpha_{2},\ldots,\alpha_{r}\rbrace$. Given a non-negative integer number $b$, we define $\mathcal{H}[[b]]$ to be the numerical semigroup generated by $\lbrace \alpha_{1},\alpha_{2},\ldots,\alpha_{r},b\rbrace$, i.e., $\mathcal{H}[[b]]=\langle \alpha_{1},\alpha_{2},\ldots,\alpha_{r},b\rangle$.\end{defn} 
 
\begin{example} \normalfont Let $\mathcal{H}=\langle 5,8,17,19\rangle$ and $\mathcal{H}^{\prime}=\mathcal{H}[[14]]$. Then, $R(\mathcal{H})=k[[t^{5},t^{8},t^{17},t^{19}]]$, and 
\begin{eqnarray*}
e(R(\mathcal{H}))=5\text{ and }F(R(\mathcal{H}))=14.
\end{eqnarray*}
Moreover, $\mathcal{H}^{\prime}=\langle 5,8,14,17\rangle$, $R(\mathcal{H}^{\prime})=k[[t^{5},t^{8},t^{14},t^{17}]]$, and
\begin{eqnarray*}
e(R(\mathcal{H}^{\prime}))=5\text{ and }F(R(\mathcal{H}^{\prime}))=12.
\end{eqnarray*}\end{example}

\begin{defn}\normalfont We call $S$ a numerical semigroup ring provided $S=R(\mathcal{H})$ for some numerical semigroup $\mathcal{H}$.\end{defn}

\section{Construction of Radical Chains}\label{CRC}

\hspace{0.5cm} Notice that $\mathcal{H}[[b]]=\mathcal{H}$ if and only if $b\in\mathcal{H}$. Suppose $\mathcal{H}$ is a numerical semigroup such that $F(\mathcal{H})>-1$. Then, $R(\mathcal{H})\neq \tilde{R}(\mathcal{H})=k[[t]]$, and we have $R(\mathcal{H})\subsetneq \End_{R(\mathcal{H})}(m)\subseteq \tilde{R}(\mathcal{H})$ (up to canonical identification), where $m$ is the maximal ideal of $R(\mathcal{H})$ (see \cite{Grauert,Hong,WV}). Set $R_{1}=R(\mathcal{H})$ and $m=m_{1}$. It is easy to see that $\End_{R_{1}}(m_{1})=R(\mathcal{K})$ for some numerical semigroup $\mathcal{K}$, where $\mathcal{H}\subsetneq \mathcal{K}$. Pick a ring $R_{2}$ such that $R_{1}\subseteq R_{2}\subseteq \End_{R_{1}}(m_{1})$. Again, it is easy to see that $R_{2}=R(\mathcal{H}^{\prime})$ for some numerical semigroup $\mathcal{H}^{\prime}$, where $\mathcal{H}\subseteq \mathcal{H}^{\prime}\subseteq \mathcal{K}$. If $R_{2}=k[[t]]$, then $R_{2}=\End_{R_{1}}(m_{1})= k[[t]]$ in which case we define $M:=R_{1}\oplus R_{2}$, and $E:=\End_{R_{1}}(M)$. If $R_{2}\neq k[[t]]$, repeat the process to obtain $R_{3}$ such that $R_{2}\subseteq R_{3}\subseteq \End_{R_{2}}(m_{2})\subseteq k[[t]]$, where $m_{2}$ is the maximal ideal of $R_{2}$. If $R_{3}=k[[t]]$, define $M:=R_{1}\oplus R_{2}\oplus R_{3}$, and $E:=\End_{R_{1}}(M)$. If $R_{3}\neq k[[t]]$, repeat the process to obtain $R_{4}$, and continue in this fashion. Notice that all the rings in our chain are numerical semigroup rings associated to some numerical semigroup, and thus are commutative, complete, local, Noetherian, reduced, and have Krull dimension 1. Moreover, since $R_{1}\subseteq R_{i}$ for all $i$, we have $\End_{R_{i}}(m_{i})=\End_{R_{1}}(m_{i})$. Of course, it is possible that $R_{1}=R_{2}=R_{3}=\ldots$. To avoid such chains we make the additional restriction that all the containments must be strict except for finitely many. Since $R_{1}$ is missing only finitely many powers of $t$, there exists an $l$ such that $R_{l}=\tilde{R}_{1}=k[[t]]$, at which time we stop the chain. This leads us to the following definition.
\begin{defn}\label{radical chains}
\normalfont Let $(R,m)$ be a commutative, complete, local, Noetherian, reduced ring with Krull dimension one such that its normalization $\tilde{R}$ is regular and $R\neq \tilde{R}$. A radical chain starting from $R$ is a chain of commutative, complete, local, Noetherian, reduced rings
\begin{eqnarray}\label{CHAIN}
(R,m)=(R_{1},m_{1})\subseteq (R_{2},m_{2})\subseteq (R_{3},m_{3})\subseteq ...\subseteq (R_{l-1},m_{l-1})\subsetneq (R_{l},m_{l}),
\end{eqnarray}
where $R_{l}=\tilde{R_{1}}$, the Krull dimension of $R_{i}$ is one for $1\leq i\leq l$, and such that $R_{i}\subseteq\End_{R_{i-1}}(m_{i-1})=\End_{R_{1}}(m_{i-1})$ for each $2\leq i\leq l$, and define 
\begin{eqnarray*}
\displaystyle E=\End_{R_{1}}(M),\text{ where }M=\bigoplus_{i=1}^{l}R_{i}.
\end{eqnarray*}
\end{defn}
$Remark~1$. Notice that all rings are allowed to be repeated in the radical chain except for the normalization $R_{l}=\tilde{R}_{1}$.

$Remark~2$. By the paragraph preceding definition \ref{radical chains}, any numerical semigroup ring has a radical chain, and every ring in the radical chain is a numerical semirgoup ring. Moreover, there are several radical chains with the same starting ring.
\begin{example} \normalfont Let $\mathcal{H}=\langle 4,5,6,7\rangle$ and $R_{1}=R(\mathcal{H})=k[[t^{4}, t^{5},t^{6},t^{7}]]$. Then,
\begin{eqnarray*}
R_{1}\subseteq k[[t^{3},t^{4},t^{5}]]\subseteq k[[t]]~\text{and}~R_{1}\subseteq k[[t^{2},t^{3}]]\subseteq k[[t]]
\end{eqnarray*}
are both radical chains starting at $R_{1}$.\end{example}
\begin{corollary}\label{theorem} Given a radical chain (\ref{CHAIN}), $\gldim(E)\leq l$.\end{corollary}
\begin{proof} This is a consequence of Theorem \ref{t1.3.1}.\end{proof} 
\hspace{0.5cm} In a radical chain (\ref{CHAIN}), suppose $R_{i}=R(\mathcal{H})$ for some numerical semigroup $\mathcal{H}$. Define $R_{i,0}=R_{i}$, and 
\begin{eqnarray*}
R_{i,j}=\left\lbrace\sum_{\text{finite}}a_{i}t^{i}:a_{i}\in k,~i\in\mathcal{H}\setminus\lbrace \beta_{1},\beta_{2},\ldots,\beta_{j}\rbrace\right\rbrace ~\text{for }1\leq j\leq r,
\end{eqnarray*}
where $0=\beta_{1}<\beta_{2}<\ldots<\beta_{r}<F(\mathcal{H})$ is a list of all the elements in $\mathcal{H}$ up to $F(\mathcal{H})$ in ascending order (notice that $F(\mathcal{H})\notin \mathcal{H}$). Moreover, $R_{i,j}$ is an ideal of $R_{i}$ for $0\leq j\leq r$ and $R_{i,1}=m_{i}$. 

\begin{example} \normalfont Let $\mathcal{H}=\langle 5,8,17,19\rangle$ and $R_{1}=R(\mathcal{H})=k[[t^{5},t^{8},t^{17},t^{19}]]$. Then, $R_{1,0}=R_{1}$, $F(R(\mathcal{H}))=14$, and $0<5<8<10<13$ is a list of all elements in $\mathcal{H}$ up to $F(R(\mathcal{H}))$ in ascending order. So 
\begin{eqnarray*}
R_{1,1}&=&\left\lbrace\sum_{\text{finite}}a_{i}t^{i}:a_{i}\in k,~i\in\mathcal{H}\setminus\lbrace 0\rbrace\right\rbrace =m_{1},~R_{1,2}=\left\lbrace\sum_{\text{finite}}a_{i}t^{i}:a_{i}\in k,~i\in\mathcal{H}\setminus\lbrace 0,5\rbrace\right\rbrace ,\\R_{1,3}&=&\left\lbrace\sum_{\text{finite}}a_{i}t^{i}:a_{i}\in k,~i\in\mathcal{H}\setminus\lbrace 0,5,8\rbrace\right\rbrace ,~R_{1,4}=\left\lbrace\sum_{\text{finite}}a_{i}t^{i}:a_{i}\in k,~i\in\mathcal{H}\setminus\lbrace 0,5,8,10\rbrace\right\rbrace ,\\R_{1,5}&=&\left\lbrace\sum_{\text{finite}}a_{i}t^{i}:a_{i}\in k,~i\in\mathcal{H}\setminus\lbrace 0,5,8,10,13\rbrace\right\rbrace .
\end{eqnarray*}
\end{example} 

\hspace{0.5cm} We now construct two radical chains with both having the same starting ring. One of these constructions maximizes the length of the radical chain (called the ``lazy" construction), while the other minimizes the length of the radical chain (called the ``greedy" construction). 

\hspace{0.5cm} Given a numerical semigroup $\mathcal{H}\neq \mathbb{N}_{0}$, let $R=R(\mathcal{H})$. Notice that $\mathcal{H}$ has a minimal generating set, say $\lbrace \alpha_{1},\alpha_{2},\ldots ,\alpha_{s}\rbrace$ written in ascending order. So $\mathcal{H}=\langle \alpha_{1},\alpha_{2},\ldots ,\alpha_{s}\rangle$ , equivalently $R=k[[t^{\alpha_{1}},t^{\alpha_{2}},\ldots ,t^{\alpha_{s}}]]$. Given a non-negative integer $b$ with $b\neq \alpha_{i}$, we define $\mathcal{H}[[b]]=\langle \alpha_{1},\alpha_{2},...,\alpha_{s},b\rangle $. Since $\gcd (\alpha_{1},\alpha_{2},\ldots ,\alpha_{s})=1$ implies that $\gcd (\alpha_{1},\alpha_{2},\ldots ,\alpha_{s},b)=1$, the set $\mathcal{H}[[b]]$ is a numerical semigroup. We define $R[[t^{b}]]=R(\mathcal{H}[[b]])$, i.e., $R[[t^{b}]]$ is the numerical semigroup ring associated to $\mathcal{H}[[b]]$. It should be noted that $\mathcal{H}\subseteq \mathcal{H}[[b]]$, and equality holds if and only if $b\in \mathcal{H}$. Set $R=R_{1}$ and  define $R_{i}=R_{i-1}[[t^{F(R_{i-1})}]]~\text{for}~i\geq 2$. Since only finitely many powers of $t$ are missing from $R_{1}$, there exists an $l\geq 2$ such that $R_{l}=k[[t]]$. In particular, we have constructed the following radical chain of rings: $R_{1}\subsetneq R_{2}\subsetneq \dots \subsetneq R_{l}=k[[t]]$. By Theorem \ref{t1.3.1}, $\gldim(E)\leq l$. The radical chain of rings just constructed, the module $M$, and the ring $E$ are said to be constructed via the ``lazy" construction.

\hspace{0.5cm} To the other extreme, let $R_{1}$ be the same ring as in the previous paragraph and define $R_{2}=\End_{R_{1}}(m_{1})$. Notice that $R_{2}$ is a numerical semigroup ring and $R_{1}\subseteq R_{2}\subseteq\tilde{R}_{1}=k[[t]]$ (see \cite{Grauert,Hong,WV}). If $R_{2}=k[[t]]$, then stop. If not, let $R_{3}=\End_{R_{1}}(m_{2})$ ($R_{3}$ is a numerical semigroup ring and $R_{2}\subseteq R_{3}\subseteq \tilde{R}_{2}=\tilde{R}_{1}=k[[t]]$). If $R_{3}=k[[t]]$, then stop. Otherwise, continue the process. Since only finitely many positive powers of $t$ are missing from $R_{1}$, there exist a natural number $l$ such that $R_{l}=k[[t]]$. In particular, $R_{i}=\End_{R_{1}}(m_{i-1})~\text{for}~2\leq i\leq l$. Since $R_{1}$ is  a numerical semigroup ring, $R_{i}$ is a numerical semigroup ring for each $1\leq i\leq l$. The radical chain of rings $R_{1}\subsetneq R_{2}\subsetneq ...\subsetneq R_{l}=k[[t]]$, the module $M$, and the ring $E$ are said to be constructed via the ``greedy" construction. By Theorem \ref{t1.3.1}, $\gldim(E)\leq l$. This is the construction given in \cite{GL}.

\section{Right Indecomposable Projective and Simple Modules Over $\End_{R}(M)$}\label{Projective section}

\hspace{0.5cm} We begin with a well known result.
\begin{thm}\label{t1.1.1} Let R be a complete local Noetherian commutative ring, and $A$ be a R-algebra which is finitely generated as an R-module. Then $\overline{A}=A/J(A)$ is a semi-simple Artinian ring, where $J(A)$ is the Jacobson radical of $A$. Suppose that $1 = e_{1}+...+e_{n}$ is a decomposition of $1\in A$ into orthogonal primitive
idempotents in $A$. Then
\begin{eqnarray*}
A=\bigoplus_{i=1}^{n}e_{i}A
\end{eqnarray*}
is a decomposition of $A$ into indecomposable right ideals of A and
\begin{eqnarray*}
\overline{A}=\bigoplus_{i=1}^{n}\overline{e}_{i}\overline{A}
\end{eqnarray*}
is a decomposition of $\overline{A}$ into minimal right ideals. Moreover, $e_{i}A\cong e_{j}A$ if and only
if $\overline{e}_{i}\overline{A}\cong \overline{e}_{j}\overline{A}$ (see \cite{IR} Theorem 6.18, 6.21 and Corollary 6.22). \end{thm}
\hspace{0.5cm} The preceding theorem says that the right indecomposable summands of $A$ are of the form $P_{i}=e_{i}A$. By definition, the $P_{i}$ are the right indecomposable projective modules over $A$. The modules $S_{i}=P_{i}/J(A)$ are the right simple modules over $A$ (as well as over the semi-simple algebra $\overline{A}$) and $P_{i}\to S_{i}\to 0$ is a projective cover. We denote the map $P_{i}\to S_{i}$ by $\pi_{i}$ (the quotient/natural map). In particular, $(P_{i},\pi_{i})$ is a projective cover for $S_{i}$.

\hspace{0.5cm} Recall that a finitely generated $R$-module $M$ is $torsion$-$free$ provided the natural map $M\rightarrow M\otimes_{R}\overline{R}$ is injective, where $\overline{R}$ is the total quotient ring of $R$. Suppose $R$ and $S$ are local, Noetherian, commutative, reduced rings, that are also complete with respect to their Jacobson radicals, respectively, and have Krull dimension 1. We say that $S$ is a $birational~extension$ of $R$ provided $R\subseteq S$ and $S$ is a finitely generated $R$-module contained in the total quotient ring $\overline{R}$ of $R$. Notice that if $S$ is a birational extension of $R$, then every finitely generated torsion-free $S$-module is a finitely generated torsion-free $R$-module, but not vice versa. The following lemma follows by clearing denominators.
\begin{lemma}\label{birational} Suppose $S$ is a birational extension of $R$. Let $C$ and $D$ be finitely generated torsion-free $S$-modules. Then $\Hom_{R}(C,D)=\Hom_{S}(C,D)$. Furthermore, if $M$ is a finitely generated torsion-free $R$-module, and $f:C\rightarrow M$ is an $R$-linear map, then the image of $f$ is an $S$-module.\end{lemma}

\hspace{0.5cm} For the remainder of this section, unless otherwise stated $(R,m)=(R_{1},m_{1})$ is a numerical semigroup ring and $R\neq k[[t]]$. Given a radical chain (\ref{CHAIN}), Theorem \ref{t1.3.1} implies that $\gldim(E)\leq l$.  We can represent $E$ as an $l\times l$ matrix. More specifically, $E_{ij}=\Hom_{R_{1}}(R_{j},R_{i})$. Given an integer $1\leq a\leq l$, the ring $R_{a}$ is a birational extension of $R_{1}$. Moreover, $R_{i}$ and $R_{j}$ are finitely generated torsion-free $R_{a}$-modules provided $a\leq i,j\leq l$. In particular, Lemma \ref{birational} implies that $\Hom_{R_{1}}(R_{j},R_{i})=\Hom_{R_{a}}(R_{j},R_{i})$ provided $a\leq i,j\leq l$. Hence, $E_{ij}=R_{i}$ for $1\leq j\leq i\leq l$. Moreover, 
\begin{eqnarray*}
(J(E))_{ij}=\begin{cases}m_{i}&\text{if}~R_{i}=R_{j}\\E_{ij}&\text{otherwise}\end{cases}\text{ (see \cite{WV})}.
\end{eqnarray*}It follows that if all the rings in a radical chain are distinct, then 
\begin{eqnarray*}
(J(E))_{ij}=\begin{cases}m_{i}&\text{if}~i=j\\E_{ij}&\text{otherwise}\end{cases}.
\end{eqnarray*}
\hspace{0.5cm} Since $E$ is an associative Noetherian ring with unity that is module finite over $R_{1}$ in its centre, the global dimension of $E$ is the supremum of the projective dimensions of the right (or left) simple $E$-modules (see \cite{Bass}, Proposition 6.7 page 125 or \cite{MR}, 7.1.14). Furthermore, by Theorem \ref{t1.1.1}, every simple right $E$-module $S_{i}$ has a projective cover $(P_{i},\pi_{i})$ and thus the category of finitely generated projective right $E$-modules is a Krull-Remak-Schmidt category (see Proposition 4.1 in \cite{Krause}). Consequently, given a simple right $E$-module $S$, the projective right $E$-modules in the projective resolution of $S$ are isomorphic to a finite direct sum of indecomposable projective modules (Krull-Remak-Schmidt Theorem).

\hspace{0.5cm} The ring $E$ has a decomposition $I_{l}=e_{1}+e_{2}+\ldots+e_{l}$ into orthogonal primitive idempotents, where $I_{l}$ is the $l\times l$ identity matrix, and $e_{i}$ is the $l\times l$ matrix with $1$ in the $ii$-th entry and zero otherwise. In particular, $E=\displaystyle\bigoplus_{i=1}^{l}e_{i}E$. Since $R_{1}$ is a complete local Noetherian commutative ring and $E$ is a finitely generated $R$-module, Theorem \ref{t1.1.1} implies that the right indecomposable projective modules of $E$ are the matrices $P_{i}=e_{i}E$. We sometimes identify $P_{i}$ with its non-zero row, that is, we think of $P_{i}$ as the $i$-th row of $E$. Furthermore, the right simple $E$-modules are $S_{i}=P_{i}/J(E)$. The maps $\pi_{i}:P_{i}\to S_{i}=P_{i}/J(E)$ are the quotient/natural maps and $(P_{i},\pi_{i})$ is a projective cover for $S_{i}$. If all the rings in a radical chain are distinct, then $S_{i}=e_{i}D_{l}$, where $D_{l}$ is the $l\times l$ diagonal matrix with diagonal entries $k$. Similar to the identification for projective modules, we sometimes identify $S_{i}$ with its non-zero row. Notice that under this identification, $P_{i}$ and $S_{i}$ are still right $E$-modules.

$Remark$. If a ring is repeated in our radical chain, then $S_{i}\neq e_{i}D_{l}$ (see next example).

\begin{example} \normalfont Let $R_{1}=k[[t^{2},t^{3}]],~R_{2}=k[[t]]$. Consider the following radical chains:
\begin{eqnarray*}
&\text{Radical chain 1:}&~R_{1}\subseteq R_{2},~M=R_{1}\oplus R_{2},~E=\End_{R_{1}}(M),\\
&\text{Radical chain 2:}&~R_{1}\subseteq R_{1}\subseteq R_{2},~M^{\prime}=R_{1}\oplus R_{1}\oplus R_{2},~E^{\prime}=\End_{R_{1}}(M^{\prime}).
\end{eqnarray*}
Then for radical chains 1 and 2, respectively, we have the following endomorphism ring, first right indecomposable projective module, Jacobson radical, and right simple module:
\begin{eqnarray*}
E=\left(\begin{matrix}R_{1}&m_{1}\\R_{2}&R_{2}\end{matrix}\right),~P_{1}&=&\left(\begin{matrix}R_{1}&m_{1}\\0&0\end{matrix}\right),~J(E)=\left(\begin{matrix}m_{1}&m_{1}\\R_{2}&m_{2}\end{matrix}\right)\Rightarrow S_{1}=\left(\begin{matrix}k&0\\0&0\end{matrix}\right)\\E^{\prime}=\left(\begin{matrix}R_{1}&R_{1}&m_{1}\\R_{2}&R_{2}&R_{2}\\R_{2}&R_{2}&R_{2}\end{matrix}\right),~P_{1}^{\prime}&=&\left(\begin{matrix}R_{1}&R_{1}&m_{1}\\0&0&0\\0&0&0\end{matrix}\right),~J(E^{\prime})=\left(\begin{matrix}m_{1}&m_{1}&m_{1}\\R_{2}&R_{2}&m_{2}\\R_{2}&R_{2}&m_{2}\end{matrix}\right)\Rightarrow S_{1}^{\prime}=\left(\begin{matrix}k&k&0\\0&0&0\\0&0&0\end{matrix}\right)
\end{eqnarray*}
We identify $P_{1}=(R_{1}~m_{1}),~S_{1}=(k~0),~P_{1}^{\prime}=(R_{1}~R_{1}~m_{1})$, and $S_{1}^{\prime}=(k~k~0)$.
\end{example}
\hspace{0.5cm} Suppose $X$ is an $E$-module which is represented by an $l\times l$ matrix. Then $X_{i}=e_{i}X$ is both an $R_{i}$-module and also a right $E$-module, and we write
\begin{eqnarray*}
\displaystyle X=\bigoplus_{i=1}^{l}X_{i}.
\end{eqnarray*}
We sometimes identify $X$ with its non-zero rows. 
\begin{example} \normalfont Suppose $l\geq 4$. Then $P_{i}=e_{i}E=E_{i}$ for $1\leq i\leq 4$, and we identify $P_{i}$ with the $i$-th row of $E$, so 
\begin{eqnarray*}
E=\left(\begin{matrix}P_{1}\\P_{2}\\P_{3}\\P_{4}\end{matrix}\right)\text{ and }
P_{1}\oplus P_{3}\oplus P_{4}=\left(\begin{matrix}P_{1}\\0\\P_{3}\\P_{4}\end{matrix}\right)\underbrace{=}_{\text{identified with}}\left(\begin{matrix}P_{1}\\P_{3}\\P_{4}\end{matrix}\right).
\end{eqnarray*}\end{example}
\hspace{0.5cm} A similar identification is used for $E$-maps $f:X\to Y$. For example, we write $f_{i}=e_{i}f$ for the $i$-th row of $f$ and identify $f_{i}$ with its non-zero row, i.e. the $i$-th row of $f$.

\hspace{0.5cm} For any $1\leq i,j\leq l$,
\begin{eqnarray*}
\Hom_{E}(P_{i},P_{j})=\Hom_{E}(e_{i}E,e_{j}E)\cong e_{j}Ee_{i}\subseteq k[[t]].
\end{eqnarray*}
Therefore, any non-zero morphism $P_{i}\to P_{j}$ is of the form $ut^{\alpha}$ for some $\alpha\in\mathbb{N}_{0}$ and $u$ a unit. Adjusting the morphism by multiplication by $u^{-1}$, an automorphism of $P_{j}$, we can assume without loss of generality that the non-zero morphisms from $P_{i}$ to $P_{j}$ are multiplication with some $t^{\alpha}$.

\section{The Functor $\lceil ~\rceil $}\label{The functor lceil rceil}

\begin{defn} \normalfont Given a radical chain (\ref{CHAIN}) and a non-negative integer $a$, we define 
\begin{eqnarray*}
\lceil a\rceil (E)=\End_{R_{1}}(\lceil a\rceil(M)),\text{ where }\lceil a\rceil(M)=\left(\bigoplus_{i=1}^{a}T_{i}\right)\oplus M\text{ with }T_{i}=R_{1},
\end{eqnarray*}
where $\lceil 0\rceil(E)=E$ and $\lceil 0\rceil(M)=M$.
\end{defn} 

$Remark$. Observe that $M$ is not a right (nor left) $E$-module. Also, for any $b\geq 1$ and $a\geq 0$, since $R_{1}\subseteq R_{b}$ we have $\End_{R_{1}}(\lceil a\rceil(M))=\End_{R_{b}}(\lceil a\rceil(M))$.

\hspace{0.5cm} We now define a functor $\lceil ~\rceil$ from the category of right $E$-modules (denoted by $\textbf{Mod}_{E}$) to the category of right $E\lceil a\rceil$-modules (denoted by $\textbf{Mod}_{E\lceil a\rceil}$). If $X$ is an $E$-module, then it can be represented as an $n\times l$ matrix. We define $\lceil a\rceil(X)$ to be the $(n+a)\times (l+a)$ matrix with the following block form:
\begin{eqnarray*}
\lceil a\rceil(X)=\left(\begin{matrix}A_{a\times a}&B_{a\times l}\\C_{n\times a}&X\end{matrix}\right) ,
\end{eqnarray*}
where $A_{a\times a}=X_{11}$ for $1\leq i,j\leq a$, $B_{ij}=X_{1j}$ for $1\leq i\leq a$ and $1\leq j\leq l$, and $C_{ij}=X_{i1}$ for $1\leq i\leq n$ and $1\leq j\leq a$. It follows that $\lceil a\rceil(X)$ is an $\lceil a\rceil(E)$-module. The composition of $\lceil a\rceil$ and $\lceil b\rceil$ is defined to be $\lceil a\rceil\lceil b\rceil:E\to\lceil a\rceil\lceil b\rceil(E)$ given by $X\mapsto \lceil a\rceil\lceil b\rceil(X):=\lceil a\rceil(\lceil b\rceil(X))$, where $\lceil a\rceil\lceil b\rceil(E)=\End_{R_{1}}(\lceil a\rceil\lceil b\rceil(M))$ and $\lceil a\rceil\lceil b\rceil(M)=\lceil a\rceil(\lceil b\rceil(M))$. Given a non-identity $E$-map $f:X\to Y$, the above construction is naturally extended to $f$ to give $\lceil a\rceil(f):\lceil a\rceil(X)\to\lceil a\rceil(Y)$. For an identity map $1_{X}:X\to X$, the matrix representation of $1_{X}$ is the $n\times n$ identity matrix, and we define $\lceil a\rceil(1_{X}):\lceil a\rceil(X)\to\lceil a\rceil(X)$ to be the $(n+a)\times (n+a)$ identity matrix. 

$Remark$. The natural extension of our construction to the $2\times 2$ identity matrix gives
\begin{eqnarray*}
\lceil 1\rceil\left(\left(\begin{matrix}1&0\\0&1\end{matrix}\right)\right)=\left(\begin{matrix}1&1&0\\1&1&0\\0&0&1\end{matrix}\right).
\end{eqnarray*}
But for functor properties to be met we need $\lceil 1\rceil(1_{X})$ to be the identity matrix, so we define it that way.

\hspace{0.5cm} If $X\stackrel{f}\to Y\stackrel{g}\to Z$ are $E$-maps and $1_{X}:X\to X$ is the identity map, then 
\begin{eqnarray*}
&\lceil a\rceil(f)&:\lceil a\rceil(X)\to\lceil a\rceil(Y)\text{ is a }\lceil a\rceil(E)\text{-map}\\&\lceil a\rceil(X)&\stackrel{\lceil a\rceil(f)}\to\lceil a\rceil(Y)\stackrel{\lceil a\rceil(g)}\to\lceil a\rceil(Z)\text{ and }\lceil a\rceil(gf)=\lceil a\rceil(g)\lceil a\rceil(f)\\&\lceil a\rceil(1_{X})&=1_{\lceil a\rceil(X)}.
\end{eqnarray*}
Hence, $\lceil a\rceil:\textbf{Mod}_{E}\to\textbf{Mod}_{\lceil a\rceil(E)}$ is a covariant functors. For $1\leq i\leq l+a$, $(\lceil a\rceil(P))_{i}=e_{i}(\lceil a\rceil(E))$ are the right indecomposable projective $\lceil a\rceil(E)$-modules and $(\lceil a\rceil(S))_{i}=(\lceil a\rceil(P))_{i}/J(\lceil a\rceil (E))$ are the right simple $\lceil a\rceil(E)$-modules. The following proposition gives a connection between the right indecomposable projective (and simple) $E$-modules and the right indecomposable projective (and simple) $\lceil a\rceil(E)$-modules. Similarly, the right indecomposable projective $\lceil a\rceil\lceil b\rceil(E)$-modules are $(\lceil a\rceil\lceil b\rceil(P))_{i}=e_{i}(\lceil a\rceil\lceil b\rceil(E))$, and the right simple $\lceil a\rceil\lceil b\rceil(E)$-modules are $(\lceil a\rceil\lceil b\rceil(S))_{i}=(\lceil a\rceil\lceil b\rceil(P))_{i}/J(\lceil a\rceil\lceil b\rceil(E))$. 

\hspace{0.5cm} Given a radical chain (\ref{CHAIN}), if $R_{i}=R_{j}$ in the radical chain, then the $i$-th and $j$-th row of $E$ are the same, and thus $P_{i}=e_{i}E\cong e_{j}E=P_{j}$. In this case $S_{i}=P_{i}/J(E)\cong P_{j}/J(E)=S_{j}$. Since the first $a+1$ rows of $\lceil a\rceil(E)$ are the same (all corresponding to $R_{1}$ in the beginning of the radical chain), we have $(\lceil a\rceil(P))_{i}\cong (\lceil a\rceil(P))_{j}$ and $(\lceil a\rceil(S))_{i}\cong (\lceil a\rceil(S))_{j}$ for $1\leq i,j\leq a+1$. Furthermore, for $i\geq 2$, the first row of $e_{i}E$ is all zeros, so $\lceil a\rceil(e_{i}E)=e_{i+a}(\lceil a\rceil(E))$. Hence, $\lceil a\rceil(P_{i})=\lceil a\rceil(e_{i}E)=e_{i+a}(\lceil a\rceil(E))=(\lceil a\rceil(P))_{i+a}$. An immediate consequence of this construction and the preceding discussion is the following results which we state as a lemma for future reference.
\begin{lemma}\label{l1} For a radical chain (\ref{CHAIN}) and using the above notation with $a\geq 0$, we have the following:
\\(a) If $X$ and $Y$ are right $E$-modules with $Y\subseteq X$, then $\lceil a\rceil(X/Y)=\lceil a\rceil(X)/\lceil a\rceil(Y)$. 
\\(b) $J(\lceil a\rceil(E))=\lceil a\rceil(J(E))$.
\\(c) $\lceil a\rceil(E)$ and $E$ are Mortia-equivalent, so their module categories are essentially the same.
\\(d) If $b$ is also a non-negative integer, then $\lceil a\rceil\lceil b\rceil(X)=\lceil a+b\rceil(X)=\lceil b\rceil\lceil a\rceil(X)$ for any $E$-module $X$.
\\(e) If $R_{i}=R_{j}$, then $P_{i}\cong P_{j}$. Consequently, $(\lceil a\rceil(P))_{i}\cong (\lceil a\rceil(P))_{j}$ for $1\leq i,j\leq a+1$.
\\(f) If $R_{i}=R_{j}$, then $S_{i}\cong S_{j}$. Consequently, $(\lceil a\rceil(S))_{i}\cong (\lceil a\rceil(S))_{j}$ for $1\leq i,j\leq a+1$.
\\(g) For $i\geq 2$, $\lceil a\rceil(P_{i})=(\lceil a\rceil(P))_{i+a}$.
\\(h) For $i\geq 2$, $\lceil a\rceil(S_{i})=(\lceil a\rceil(S))_{i+a}$.
\\(i) $\lceil a\rceil(P_{1})=\displaystyle\bigoplus_{i=1}^{a+1}(\lceil a\rceil(P))_{i}$. Consequently, $\lceil a\rceil(S_{1})=\displaystyle\bigoplus_{i=1}^{a+1}(\lceil a\rceil(S))_{i}$.
\\(j) $\displaystyle\lceil a\rceil\left(\bigoplus_{i=1}^{c}Q_{i}\right)=\bigoplus_{i=1}^{c}\lceil a\rceil(Q_{i})$, where $Q_{i}\in\lbrace P_{1},P_{2},P_{3},\ldots,P_{l}\rbrace$.
\\(k) $(\lceil a\rceil\lceil b\rceil(P))_{a+b+i}=\lceil a\rceil\lceil b\rceil(P_{i})$ and $(\lceil a\rceil\lceil b\rceil(S))_{a+b+i}=\lceil a\rceil\lceil b\rceil(S_{i})$.
\end{lemma}
\hspace{0.5cm} When $a=1$, the module $\lceil 1\rceil(E)$ has the following matrix block form;
\begin{eqnarray*}
\lceil 1\rceil(E) &=&\left(\begin{matrix}R_{1}&M^{\ast}\\(\Hom_{R_{1}}(R_{1},M))^{T}&E\end{matrix}\right)=\left(\begin{matrix}R_{1}&M^{\ast}\\M^{T}&E\end{matrix}\right),
\end{eqnarray*}
where $M^{\ast}=\Hom_{R_{1}}(M,R_{1})=\displaystyle\Hom_{R_{1}}\left(\bigoplus_{i=1}^{l}R_{i},R_{1}\right)\cong \bigoplus_{i=1}^{l}\Hom_{R_{1}}(R_{i},R_{1})$ and
\begin{eqnarray*}
M^{T}=\left(\begin{matrix}R_{1}\\R_{2}\\\vdots\\R_{l}\end{matrix}\right).
\end{eqnarray*}

\begin{example}\label{ex1} \normalfont Let $R_{1}=k[[t^{3},t^{4},t^{5}]],~R_{2}=k[[t^{2},t^{3}]]$, $R_{3}=k[[t]]$ and let $m_{1},~m_{2}$, and $m_{3}$ be their maximal ideals, respectively. If $M=R_{1}\oplus R_{2}\oplus R_{3}$ and $E=\End_{R_{1}}(M)$, then
\begin{eqnarray*}
E=\left(\begin{matrix}R_{1}&m_{1}&m_{1}\\R_{2}&R_{2}&m_{2}\\R_{3}&R_{3}&R_{3}\end{matrix}\right),~\lceil 1\rceil(E)=\left(\begin{matrix}R_{1}&M^{\ast}\\M^{T}&E\end{matrix}\right),\text{ where }M^{\ast}=\left(\begin{matrix}R_{1}&m_{1}&m_{1}\end{matrix}\right)\text{ and }M^{T}=\left(\begin{matrix}R_{1}\\R_{2}\\R_{3}\end{matrix}\right).
\end{eqnarray*}
It follows that
\begin{eqnarray*}
J(E)&=&\left(\begin{matrix}m_{1}&m_{1}&m_{1}\\R_{2}&m_{2}&m_{2}\\R_{3}&R_{3}&m_{3}\end{matrix}\right),~J(\lceil 1\rceil(E))=\left(\begin{matrix}m_{1}&m_{1}&m_{1}&m_{1}\\m_{1}&m_{1}&m_{1}&m_{1}\\R_{2}&R_{2}&m_{2}&m_{2}\\R_{3}&R_{3}&R_{3}&m_{3}\end{matrix}\right)=\lceil 1\rceil(J(E)).
\end{eqnarray*}
This gives us\begin{small}
\begin{eqnarray*}
(\lceil 1\rceil(P))_{1}\oplus(\lceil 1\rceil(P))_{2}&=&\left(\begin{matrix}R_{1}&R_{1}&m_{1}&m_{1}\\0&0&0&0\\0&0&0&0\\0&0&0&0\end{matrix}\right)\oplus\left(\begin{matrix}0&0&0&0\\R_{1}&R_{1}&m_{1}&m_{1}\\0&0&0&0\\0&0&0&0\end{matrix}\right)=\left(\begin{matrix}R_{1}&R_{1}&m_{1}&m_{1}\\R_{1}&R_{1}&m_{1}&m_{1}\\0&0&0&0\\0&0&0&0\end{matrix}\right)=(\lceil 1\rceil(P_{1})),\\(\lceil 1\rceil(P))_{3}&=&\left(\begin{matrix}0&0&0&0\\0&0&0&0\\R_{2}&R_{2}&R_{2}&m_{2}\\0&0&0&0\end{matrix}\right)=\lceil 1\rceil(P_{2}),\\(\lceil 1\rceil(S))_{1}\oplus(\lceil 1\rceil(S))_{2}&=&\left(\begin{matrix}k&k&0&0\\0&0&0&0\\0&0&0&0\\0&0&0&0\end{matrix}\right)\oplus\left(\begin{matrix}0&0&0&0\\k&k&0&0\\0&0&0&0\\0&0&0&0\end{matrix}\right)=\left(\begin{matrix}k&k&0&0\\k&k&0&0\\0&0&0&0\\0&0&0&0\end{matrix}\right)=\lceil 1\rceil(S_{1})\\(\lceil 1\rceil(S))_{3}&=&\left(\begin{matrix}0&0&0&0\\0&0&0&0\\0&0&k&0\\0&0&0&0\end{matrix}\right)=\lceil 1\rceil(S_{2})\\(\lceil 1\rceil(P_{2}))\oplus (\lceil 1\rceil(P_{3}))&=&\left(\begin{matrix}0&0&0&0\\0&0&0&0\\R_{2}&R_{2}&R_{2}&m_{2}\\0&0&0&0\end{matrix}\right)\oplus\left(\begin{matrix}0&0&0&0\\0&0&0&0\\0&0&0&0\\R_{3}&R_{3}&R_{3}&R_{3}\end{matrix}\right)=\left(\begin{matrix}0&0&0&0\\0&0&0&0\\R_{2}&R_{2}&R_{2}&m_{2}\\R_{3}&R_{3}&R_{3}&R_{3}\end{matrix}\right)=\lceil 1\rceil(P_{2}\oplus P_{3}).
\end{eqnarray*}\end{small}
The maps $\pi_{1}:P_{1}\to S_{1}$, $\lceil 1\rceil(\pi_{1}):\lceil 1\rceil(P_{1})\to\lceil 1\rceil(S_{1})$, $(\lceil 1\rceil(\pi))_{1}:(\lceil 1\rceil(P))_{1}\to\lceil 1\rceil(S))_{1}$ are all quotient maps (but of course, all distinct). To see that $\lceil 1\rceil(\pi_{1})$ and $(\lceil 1\rceil(\pi))_{1}$ are distinct, the former map is quotient by $m_{1}$ across the first two rows and the zero map every where else. While the latter map is quotient by $m_{1}$ in the first row and the zero map in every where else (also one can observe they have distinct domains and co-domains).
\end{example} 
Many of the proofs below are done by identifying modules with their non-zero row(s), the map $\pi_{i}$ with its action on the $i$-th row of $P_{i}$, and all other maps with the matrix obtained from removing their zero row(s) and column(s). We give an example illustrating why this is so useful.
\begin{example}\label{id} \normalfont Let $R_{1},~R_{2}$, and $R_{3}$ be the rings given in example \ref{ex1}. Then we have the following exact sequence, $0\leftarrow S_{2}\stackrel{\pi_{2}}\longleftarrow P_{2}\stackrel{\gamma}\longleftarrow P_{1}\oplus P_{3}\stackrel{\delta}\longleftarrow P_{3}\leftarrow 0$. Writing the modules and maps as matrices gives;
\begin{eqnarray*}
0\leftarrow \left(\begin{matrix}0&0&0\\0&k&0\\0&0&0\end{matrix}\right)\stackrel{\pi_{2}}\longleftarrow \left(\begin{matrix}0&0&0\\R_{2}&R_{2}&m_{2}\\0&0&0\end{matrix}\right)\stackrel{\gamma}\longleftarrow\left(\begin{matrix}R_{1}&m_{1}&m_{1}\\0&0&0\\R_{3}&R_{3}&R_{3}\end{matrix}\right)\stackrel{\delta}\longleftarrow\left(\begin{matrix}0&0&0\\0&0&0\\R_{3}&R_{3}&R_{3}\end{matrix}\right)\leftarrow 0,
\end{eqnarray*}
where 
\begin{eqnarray*}
\gamma=\left(\begin{matrix}0&0&0\\1&0&t^{2}\\0&0&0\end{matrix}\right),\delta=\left(\begin{matrix}0&0&t^{3}\\0&0&0\\0&0&-t\end{matrix}\right).
\end{eqnarray*}
We identify the above modules and maps to the following modules and maps, respectively;
\begin{eqnarray*}
0\leftarrow \left(\begin{matrix}0&k&0\end{matrix}\right)\stackrel{\pi_{2}}\longleftarrow \left(\begin{matrix}R_{2}&R_{2}&m_{2}\end{matrix}\right)\stackrel{\gamma}\longleftarrow\left(\begin{matrix}R_{1}&m_{1}&m_{1}\\R_{3}&R_{3}&R_{3}\end{matrix}\right)\stackrel{\delta}\longleftarrow\left(\begin{matrix}R_{3}&R_{3}&R_{3}\end{matrix}\right)\leftarrow 0,
\end{eqnarray*}
where $\pi_{2}$ is identified with its action on the second row of $P_{2}$, $\gamma=\left(\begin{matrix}1&t^{2}\end{matrix}\right),\delta=\left(\begin{matrix}t^{3}\\-t\end{matrix}\right),\ker\pi_{2}=\left(\begin{matrix}R_{2}&m_{2}&m_{2}\end{matrix}\right)=J(P_{2})$. If we apply $\lceil 1\rceil$ to the above exact sequence we get
\begin{eqnarray}\label{not a complex}
0\leftarrow \lceil 1\rceil(S_{2})\stackrel{\lceil 1\rceil(\pi_{2})}\longleftarrow \lceil 1\rceil(P_{2})\stackrel{\lceil 1\rceil(\gamma)}\longleftarrow\lceil 1\rceil(P_{1}\oplus P_{3})\stackrel{\lceil 1\rceil(\delta)}\longleftarrow\lceil 1\rceil(P_{3})\leftarrow 0,
\end{eqnarray}
where under this identification 
\begin{eqnarray*}
\lceil 1\rceil(S_{2})&=&\left(\begin{matrix}0&0&k&0\end{matrix}\right),\lceil 1\rceil(P_{2})=\left(\begin{matrix}R_{2}&R_{2}&R_{2}&m_{2}\end{matrix}\right), \lceil 1\rceil(P_{3})=\left(\begin{matrix}R_{3}&R_{3}&R_{3}&R_{3}\end{matrix}\right)\\\lceil 1\rceil(P_{1}\oplus P_{3})&=&\left(\begin{matrix}R_{1}&R_{1}&m_{1}&m_{1}\\R_{1}&R_{1}&m_{1}&m_{1}\\R_{3}&R_{3}&R_{3}&R_{3}\end{matrix}\right),\lceil 1\rceil(\gamma)=\left(\begin{matrix}1&1&t^{2}\end{matrix}\right),~\lceil 1\rceil(\delta)=\left(\begin{matrix}t^{3}\\t^{3}\\-t\end{matrix}\right).
\end{eqnarray*}
Notice that (\ref{not a complex}) is not even a complex let alone exact. It is true that $\ker(\lceil 1\rceil(\pi_{2}))=\left(\begin{matrix}R_{2}&R_{2}&m_{2}&m_{2}\end{matrix}\right)=\text{Im}(\lceil 1\rceil(\gamma))$, however, $\text{Im}(\lceil 1\rceil(\delta))\nsubseteq\ker(\lceil 1\rceil(\gamma)$. However, if $P_{1}$ does not appear in the exact sequence of $S_{i}$, then exactness is preserved by $\lceil a\rceil$. For example, we have the following exact sequence (up to identification);
\begin{eqnarray*}
0\leftarrow \left(\begin{matrix}0&0&k\end{matrix}\right)\stackrel{\pi_{3}}\longleftarrow \left(\begin{matrix}R_{3}&R_{3}&R_{3}\end{matrix}\right)\stackrel{(1~t)}\longleftarrow\left(\begin{matrix}R_{2}&R_{2}&m_{2}\\R_{3}&R_{3}&R_{3}\end{matrix}\right)\stackrel{\left(\begin{matrix}t^{2}\\-t\end{matrix}\right)}\longleftarrow\left(\begin{matrix}R_{3}&R_{3}&R_{3}\end{matrix}\right)\leftarrow 0,
\end{eqnarray*}
and applying $\lceil 1\rceil$ to this sequence gives the following exact sequence,
\begin{eqnarray*}
0\leftarrow \left(\begin{matrix}0&0&0&k\end{matrix}\right)\stackrel{\lceil 1\rceil(\pi_{3})}\longleftarrow \left(\begin{matrix}R_{3}&R_{3}&R_{3}&R_{3}\end{matrix}\right)\stackrel{(1~t)}\longleftarrow\left(\begin{matrix}R_{2}&R_{2}&R_{2}&m_{2}\\R_{3}&R_{3}&R_{3}&R_{3}\end{matrix}\right)\stackrel{\left(\begin{matrix}t^{2}\\-t\end{matrix}\right)}\longleftarrow\left(\begin{matrix}R_{3}&R_{3}&R_{3}&R_{3}\end{matrix}\right)\leftarrow 0.
\end{eqnarray*}
\end{example}
\hspace{0.5cm} It is well known that a projective resolution
\begin{eqnarray*}
0\leftarrow M\stackrel{\varepsilon}\longleftarrow Q_{0}\stackrel{d_{1}}\longleftarrow Q_{1}\stackrel{d_{2}}\longleftarrow \ldots \stackrel{d_{n}}\longleftarrow Q_{n}\leftarrow 0
\end{eqnarray*}
is minimal if and only if $\text{Im}(d_{i})\subseteq J(Q_{i-1})$ (the Jacobson radical of $P_{i}$) for $i=1,2,\ldots,n$ and $P_{0}\stackrel{\varepsilon}\to M\to 0$ is a projective cover. The construction of the functor $\lceil a\rceil$ and the preceding statements proves the following useful proposition.
\begin{proposition}\label{exactness} Given a radical chain (\ref{CHAIN}), suppose 
\begin{eqnarray}\label{mod}
0\leftarrow M\stackrel{\varepsilon}\longleftarrow L_{0}\stackrel{d_{1}}\longleftarrow L_{1}\stackrel{d_{2}}\longleftarrow \ldots \stackrel{d_{n}}\longleftarrow L_{n}\leftarrow 0
\end{eqnarray}
is an exact sequence of $E$-modules. 
\\(a) If $P_{1}$ is not a direct summand of $L_{i}$ for $i=0,1,2,\ldots,n$, then for any non-negative integer $a$ we have the following exact sequence:
\begin{eqnarray*}
0\leftarrow \lceil a\rceil(M)\stackrel{\lceil a\rceil(\varepsilon)}\longleftarrow \lceil a\rceil(L_{0})\stackrel{\lceil a\rceil(d_{1})}\longleftarrow \lceil a\rceil(L_{1})\stackrel{\lceil a\rceil(d_{2})}\longleftarrow \ldots \stackrel{\lceil a\rceil(d_{n})}\longleftarrow \lceil a\rceil(L_{n})\leftarrow 0
\end{eqnarray*}
(b) For a fixed $1\leq i\leq l$, if $M=S_{i},~L_{0}=P_{i}$, and $\varepsilon=\pi_{i}$ in (\ref{mod}), then for any non-negative integer $a$ we have $\lceil a\rceil(\ker\pi_{i})=\ker(\lceil a\rceil(\pi_{i}))$ and $\lceil a\rceil(\text{Im}(d_{1}))=\text{Im}(\lceil a\rceil(d_{1}))$. In particular, $\ker(\lceil a\rceil(\pi_{i}))=\text{Im}(\lceil a\rceil(d_{1}))$ for any non-negative integer $a$.
\\(c) Given $1\leq j\leq n$, if $L_{j-1}$ and $L_{j}$ are direct sum of indecomposable right $E$-modules and $P_{1}$ is not a direct summand of $L_{j-1}$ nor of $L_{j}$ and $\text{Im}(d_{j})\subseteq J(L_{j-1})$, then for any non-negative integer $a$ we have $\lceil a\rceil\text{Im}(d_{j})=\text{Im}(\lceil a\rceil(d_{j}))\subseteq J(\lceil a\rceil(L_{j-1}))$. Moreover, under the identification in Example \ref{id}, $\lceil a\rceil(d_{j})=d_{j}$.   
\\(d) Consequently, if (\ref{mod}) is a minimal projective resolution of $M=S_{i}$, where $i\geq 2$, and $P_{1}$ is not a direct summand of $L_{j}$ for $0\leq j\leq n$, then 
\begin{eqnarray*}
0\leftarrow \lceil a\rceil(S_{i})\stackrel{\lceil a\rceil(\pi_{i})}\longleftarrow \lceil a\rceil(L_{0})\stackrel{\lceil a\rceil(d_{1})}\longleftarrow \lceil a\rceil(L_{1})\stackrel{\lceil a\rceil(d_{2})}\longleftarrow \ldots \stackrel{\lceil a\rceil(d_{n})}\longleftarrow \lceil a\rceil(L_{n})\leftarrow 0
\end{eqnarray*}
is a minimal $\lceil a\rceil(E)$-projective resolution of $\lceil a\rceil(S_{i})$.
\end{proposition}

\section{Family of Starting Rings}\label{Family of Starting Rings}

\hspace{0.5cm} Fix an even integer $n\geq 6$, and pick an integer $\frac{3n}{2}+1\leq a\leq 2n-1$. Define 
\begin{eqnarray*}
A_{n}^{a}(1)=\left\lbrace 0,n,\frac{3n}{2}+w;w=0,1,2,\ldots ,n-1\right\rbrace\text{~(this ring only depends on n)},
\end{eqnarray*}
and for each natural number $i\geq 2$, define
\begin{eqnarray*}
A_{n}^{a}(i)=\lead \left\lbrace 0,\dfrac{jn}{2},a+1+(i-2)\dfrac{n}{2}+w;j=2,3,\ldots ,i+1,~w=0,1,\ldots ,n-1\right\rbrace .
\end{eqnarray*}
Let $\mathcal{H}_{n}^{a}(i)$ be the numerical semigroup generated by $A_{n}^{a}(i)$, i.e. $\mathcal{H}_{n}^{a}(i)=\langle A_{n}^{a}(i)\rangle$. Notice that $F(\mathcal{H}_{n}^{a}(i))=F(\mathcal{H}_{n}^{a}(i-1))+\frac{n}{2}$ for each natural number $i\geq 3$. When $a$ and $n$ are understood, we write $A(i)$ and $\mathcal{H}(i)$ for $A_{n}^{a}(i)$ and $\mathcal{H}_{n}^{i}$, respectively. In this case, we let $R^{i}=R(\mathcal{H}(i)))$. For each $i\in \mathbb{N}$, we construct a radical chain starting from $R^{i}$:
\begin{eqnarray}\label{radical chain of family of starting rings}
R^{i}=R^{i}_{1}\subseteq R^{i}_{2}\subseteq \ldots \subseteq R^{i}_{l_{i}}=k[[t]],
\end{eqnarray}
and we call $\mathcal{F}(n,a):=\lbrace R^{i}:i\in\mathbb{N}\rbrace$ a family of starting rings. We define
\begin{eqnarray*}
E^{i}=\End_{R^{i}_{1}}(M^{i}),~\text{where}~M^{i}=\bigoplus_{j=1}^{l_{i}}R^{i}_{j}.
\end{eqnarray*}
The indecomposable projective right $E^{i}$-modules are denoted by $P^{i}_{1},P^{i}_{2},\ldots,P^{i}_{l_{i}}$, i.e. $P^{i}_{j}=e_{j}E^{i}$. Similarly, the simple right $E^{i}$-modules are denoted by $S^{i}_{1},S^{i}_{2},\ldots,S^{i}_{l_{i}}$, i.e. $S^{i}_{j}=P^{i}_{j}/J(E^{i})$. By Theorem \ref{t1.3.1}, $2\leq \gldim(E^{i})\leq l_{i}$. It should be noted that different constructions of the radical chain (\ref{radical chain of family of starting rings}) give rise to different $E^{i}$. So for each $i$, we must first decide which construction to apply to get the radical chain (\ref{radical chain of family of starting rings}). The following notation will be very useful throughout the remainder of this paper.
\begin{notation} Let $\varepsilon =a+1-\frac{3n}{2},~\varepsilon_{1}=a+1-n,~\varepsilon_{2}=a+1-\frac{n}{2}$, $\zeta =(t^{n}~t^{\frac{3n}{2}})$, and
\begin{eqnarray*}
\tau =\left ( \begin{array}{cc}
t^{\frac{3n}{2}}&t^{2n}\\
-t^{n}&-t^{\frac{3n}{2}}\end{array} \right) :=\left ( \begin{array}{c}
\tau_{1}\\
\tau_{2}\end{array} \right),~\phi =\left ( \begin{array}{c}
t^{\varepsilon_{1}}\\
-t^{\varepsilon}\end{array} \right) ,\eta =\left( \begin{array}{c}
t^{\varepsilon_{2}}\\
-t^{\varepsilon_{1}}\end{array} \right) ,~\sigma =\left( \begin{array}{c}
t^{\frac{3n}{2}}\\
-t^{n}\end{array} \right) ,~\mu =\left( \begin{array}{c}t^{\frac{n}{2}}\\-1\end{array} \right)
\end{eqnarray*}\end{notation}

\subsection{Constructing Endomorphism Rings of Large Global Dimension}\label{Constructing Endomorphism Rings of Large Global Dimension.}

Throughout this section, we assume the radical chain (\ref{radical chain of family of starting rings}), the module $M^{i}$, and the ring $E^{i}$ are constructed via the lazy construction for each $i\in \mathbb{N}$. Observe that $R^{1}_{1}=R^{2}_{a+1-\frac{3n}{2}}$, and $R^{i+1}_{j+\frac{n}{2}-1}=R^{i}_{j}$ for $i\geq 2$ and $1\leq j\leq l_{i}$. The following proposition is a direct consequence of this observation. 
\begin{proposition}\label{p2.4.2} Using the notation introduced at the beginning of this section,
\\(a) $l_{1}=\dfrac{3n}{2}-1$, $l_{2}=a-2$, and $l_{i+1}=l_{i}+\dfrac{n}{2}-1$ for $i\geq 2$.
\\(b) For all $i\geq 2$ and $2\leq j\leq l_{i}$, we have 
\begin{eqnarray*}
\left(\left\lceil\frac{n}{2}-1\right\rceil(P^{i})\right)_{j+\frac{n}{2}-1}&=&\left\lceil \frac{n}{2}-1 \right\rceil(P^{i}_{j})=P_{j+\frac{n}{2}-1}^{i+1},\\\left(\left\lceil\frac{n}{2}-1\right\rceil(S^{i})\right)_{j+\frac{n}{2}-1}&=&\left \lceil \frac{n}{2}-1\right \rceil(S^{i}_{j})\cong S^{i+1}_{j+\frac{n}{2}-1}\text{ as }E^{i+1}\text{-modules},
\end{eqnarray*}
where $\left(\left\lceil\frac{n}{2}-1\right\rceil(P^{i})\right)_{j+\frac{n}{2}-1}$ and $\left(\left\lceil\frac{n}{2}-1\right\rceil(S^{i})\right)_{j+\frac{n}{2}-1}$ are the $j+\frac{n}{2}-1$ indecomposable projective and simple right $\left\lceil\frac{n}{2}-1\right\rceil(E^{i})$-modules, respectively.
\end{proposition}
Unless otherwise stated in the calculations below we are using the identification used in Example \ref{id}. A quick (but tedious) calculation proves the following proposition. 
\begin{proposition}\label{p2.4.3} (a) The minimal $E^{1}$-projective resolution of $S_{j}^{1}$ are as follows:
\begin{eqnarray*}
0&\leftarrow &S_{1}^{1}\stackrel{\pi_{1}^{1}}\longleftarrow P_{1}^{1}\stackrel{t^{n}}\longleftarrow P_{n}^{1}\leftarrow 0,\\0&\leftarrow &S_{j}^{1}\stackrel{\pi_{j}^{1}}\longleftarrow P_{j}^{1}\stackrel{(1~t^{n})}\longleftarrow P_{j-1}^{1}\oplus P_{n+j-1}^{1}\stackrel{\left( \begin{matrix}t^{n}\\-1\end{matrix}\right)}\longleftarrow P_{n+j-2}^{1}\leftarrow 0~\text{for}~2\leq j\leq \frac{n}{2}\\0&\leftarrow &S_{j}^{1}\stackrel{\pi_{j}^{1}}\longleftarrow P_{i}^{1}\stackrel{(1~t^{\frac{3n}{2}-j})}\longleftarrow P_{j-1}^{1}\oplus P_{\frac{3n}{2}-1}^{1}\stackrel{\left( \begin{matrix}t^{\frac{3n}{2}-j+1}\\-t\end{matrix}\right)}\longleftarrow P_{\frac{3n}{2}-1}^{1}\leftarrow 0~\text{for}~\frac{n}{2}+1\leq j\leq \frac{3n}{2}-1=l_{1}.
\end{eqnarray*}  
In particular, $\gldim(E^{1})=2$.
\\(b) The minimal $E^{2}$-projective resolutions of $S_{1}^{2}$ is 
\begin{eqnarray*}
0\leftarrow S_{1}^{2}\stackrel{\pi_{1}^{2}}\longleftarrow P_{1}^{2}\stackrel{\zeta}\longleftarrow P_{n-1}^{2}\oplus P_{\frac{3n}{2}-1}^{2}\stackrel{\phi}\longleftarrow P_{l_{2}}^{2}\leftarrow 0.
\end{eqnarray*} 
In particular, $\Pd_{E^{2}}(S_{1}^{2})=2$.
\\(c) If $q\geq 1$, then any non-zero row of $\lceil \frac{3n}{2}-3\rceil(J(P_{1}^{3q-1}))$ is the same, and denote any such row by $N^{3q-1}$. Moreover, $N^{3q-1}$ is an $E^{3q+2}$-module and
\begin{eqnarray*}
0\leftarrow S_{1}^{3q+2}\stackrel{\pi_{1}^{3q+2}}\longleftarrow P_{1}^{3q+2}\stackrel{\zeta}\longleftarrow P_{n-1}^{3q+2}\oplus P_{\frac{3n}{2}-2}^{3q+2}\stackrel{\mu}\longleftarrow N^{3q-1}\leftarrow 0.
\end{eqnarray*}
is an exact sequence of $E^{3q+2}$-modules.
\end{proposition}

Now we are in position to prove the first main result. 

\begin{thm}\label{t2.4.1} If $q\geq 0$, we have the following: 
\begin{eqnarray*}
0\leftarrow S_{1}^{3q+2}\stackrel{d_{0}}\longleftarrow W_{0}\stackrel{d_{1}}\longleftarrow W_{1}\stackrel{d_{2}}\longleftarrow W_{2}\stackrel{d_{3}}\longleftarrow \cdots \stackrel{d_{q+1}}\longleftarrow W_{q+1}\stackrel{d_{q+2}}\longleftarrow W_{q+2}\leftarrow 0
\end{eqnarray*}
is a minimal $E^{3q+2}$-projective resolution for $S_{1}^{3q+2}$, where
\begin{eqnarray*}
W_{j}&=&\begin{cases}P_{1}^{3q+2}&~\text{if}~j=0\\\\P^{3q+2}_{(n-1)+3(j-1)(\frac{n}{2}-1)}\oplus P^{3q+2}_{(n-1)+3(j-1)(\frac{n}{2}-1)+(\frac{n}{2}-1)}&~\text{if}~j=1,2,\ldots ,q\\\\P_{(n-1)+3q(\frac{n}{2}-1)}^{3q+2}\oplus P^{3q+2}_{(n-1)+3q(\frac{n}{2}-1)+\frac{n}{2}}&~\text{if}~j=q+1\\\\P^{3q+2}_{l_{3q+2}}&~\text{if}~j=q+2\end{cases}\\d_{j}&=&\begin{cases}\pi_{1}^{3q+2}&~\text{if}~j=0\\\zeta &~\text{if}~j=1\\\tau &~\text{if}~j=2,\ldots ,q+1\\\phi &~\text{if}~j=q+2\end{cases}
\end{eqnarray*}
In particular, $\Pd_{E^{3q+2}}(S_{1}^{3q+2})=q+2$ for $q\in \mathbb{N}_{0}$. Therefore, $q+2\leq \gldim(E^{3q+2})\leq l_{3q+2}$ for $q\in \mathbb{N}_{0}$.\end{thm}

\begin{proof} We proceed by induction on $q$. The case $q=0$ is Proposition \ref{p2.4.3} (b). Assume the result holds for $q-1$ (with $q\geq 1$). By Proposition \ref{p2.4.3} (c), the following sequence of $E^{3q+2}$-modules is exact
\begin{eqnarray}\label{eq.t2}
\begin{tikzpicture}
  \matrix (m) [matrix of math nodes,row sep=3em,column sep=2em,minimum width=2em]
  {
      0 & S_{1}^{3q+2}& P_{1}^{3q+2} & P_{n-1}^{3q+2}\oplus P_{\frac{3n}{2}-2}^{3q+2} & N^{3q-1} & 0, \\};
  \path[-stealth]
    (m-1-2) edge node [left] {$ $} (m-1-1)
    (m-1-3) edge node [above] {$\pi_{1}^{3q+2}$} (m-1-2)
    (m-1-4) edge node [above] {$\zeta$} (m-1-3)
    (m-1-5) edge node [above] {$\mu$} (m-1-4)
    (m-1-6) edge node [left] {$ $} (m-1-5);
\end{tikzpicture}
\end{eqnarray}
where $N^{3q-1}$ is any non-zero row of $\lceil \frac{3n}{2}-3\rceil(J(P_{1}^{3q-1}))$. By induction, $\Pd_{E^{3q-1}}(S_{1}^{3q-1})=(q-1)+2=q+1$ (since $S_{1}^{3(q-1)+2}=S_{1}^{3q-1}$) and 
\begin{eqnarray}\label{eq.t3}\begin{tikzpicture}
  \matrix (m) [matrix of math nodes,row sep=2em,column sep=2em,minimum width=2em]
  {
      0 & S_{1}^{3q-1} & L_{0} & L_{1} & L_{2} & \cdots & L_{q} & L_{q+1} & 0, \\};
  \path[-stealth]
    (m-1-2) edge node [left] {$ $} (m-1-1)
    (m-1-3) edge node [above] {$f_{0}$} (m-1-2)
    (m-1-4) edge node [above] {$f_{1}$} (m-1-3)
    (m-1-5) edge node [above] {$f_{2}$} (m-1-4)
    (m-1-6) edge node [above] {$f_{3}$} (m-1-5)
    (m-1-7) edge node [above] {$f_{q}$} (m-1-6)
    (m-1-8) edge node [above] {$f_{q+1}$} (m-1-7)
    (m-1-9) edge node [above] {$ $} (m-1-8);
\end{tikzpicture}\end{eqnarray}
is a minimal $E^{3q-1}$-projective resolution of $S_{1}^{3q-1}$, where 
\begin{eqnarray*}
L_{j}&=&\begin{cases}P_{1}^{3q-1}&~\text{if}~j=0\\\\P^{3q-1}_{(n-1)+3(j-1)(\frac{n}{2}-1)}\oplus P^{3q-1}_{(n-1)+3(j-1)(\frac{n}{2}-1)+(\frac{n}{2}-1)}&~\text{if}~j=1,2,...,q-1\\\\P_{(n-1)+3(q-1)(\frac{n}{2}-1)}^{3q-1}\oplus P^{3q-1}_{(n-1)+3(q-1)(\frac{n}{2}-1)+\frac{n}{2}}&~\text{if}~j=q\\\\P^{3q-1}_{l_{3q-1}}&~\text{if}~j=q+1\end{cases}\\f_{j}&=&\begin{cases}\pi_{1}^{3q-1}&~\text{if}~j=0\\\zeta &~\text{if}~j=1\\\tau &~\text{for}~j=2,...,q\\\phi &~\text{if}~j=q+1\end{cases}
\end{eqnarray*}
Since $\Im (f_{1})=\ker (f_{0})=J(P_{1}^{3q-1})$, the exact sequence in (\ref{eq.t3}) yields the following exact sequence of $E^{3q-1}$-modules:
\begin{eqnarray}\label{eq.t4}\begin{tikzpicture}
  \matrix (m) [matrix of math nodes,row sep=2em,column sep=2em,minimum width=2em]
  {
      0 & J(P_{1}^{3q-1}) & L_{1} & \cdots & L_{q} & L_{q+1} & 0. \\};
  \path[-stealth]
    (m-1-2) edge node [left] {$ $} (m-1-1)
    (m-1-3) edge node [above] {$f_{1}$} (m-1-2)
    (m-1-4) edge node [above] {$f_{2}$} (m-1-3)
    (m-1-5) edge node [above] {$f_{q}$} (m-1-4)
    (m-1-6) edge node [above] {$f_{q+1}$} (m-1-5)
    (m-1-7) edge node [above] {$ $} (m-1-6);
\end{tikzpicture}\end{eqnarray}

\noindent Observe that $P_{1}^{3q-1}$ is not a direct summand of $L_{j}$ for $j=1,2,\ldots,q+1$, so applying $\lceil\frac{3n}{2}-3\rceil$ to the exact sequence (\ref{eq.t4}) and using Proposition \ref{exactness}(a) gives the following exact sequence of $\lceil\frac{3n}{2}-3\rceil(E^{3q-1})$-modules: 
\begin{eqnarray}\begin{tikzpicture}
  \matrix (m) [matrix of math nodes,row sep=3em,column sep=2em,minimum width=1em]
  {
     0 & \left \lceil \frac{3n}{2}-3\right \rceil(J(P_{1}^{3q-1})) & \left \lceil \frac{3n}{2}-3\right \rceil(L_{1}) & \cdots & \left \lceil \frac{3n}{2}-3\right \rceil(L_{q}) & \left \lceil \frac{3n}{2}-3\right \rceil(L_{q+1}) & 0, \\};
  \path[-stealth]
    (m-1-2) edge node [left] {$ $} (m-1-1)
    (m-1-3) edge node [above] {$g_{1}$} (m-1-2)
    (m-1-4) edge node [above] {$g_{2}$} (m-1-3)
    (m-1-5) edge node [above] {$g_{q}$} (m-1-4)
    (m-1-6) edge node [above] {$g_{q+1}$} (m-1-5)
    (m-1-7) edge node [above] {$ $} (m-1-6); 
\end{tikzpicture}\end{eqnarray}
where $h_{1}=\left(\begin{matrix}f_{1}\\\vdots\\f_{1}\end{matrix}\right)=\left(\begin{matrix}\zeta\\\vdots\\\zeta\end{matrix}\right)$ is a $(\lceil\frac{3n}{2}-3\rceil)\times 1$ matrix, and 
\begin{eqnarray*}
g_{j}=\left\lceil \dfrac{3n}{2}-3\right\rceil(f_{j})=\begin{cases}\left(\begin{matrix}h_{1}\\f_{1}\end{matrix}\right)&\text{if}j=1\\f_{j}&\text{if}j=2,3,\ldots, q+1\end{cases}=\begin{cases}\left(\begin{matrix}h_{1}\\\zeta\end{matrix}\right)&\text{ if }j=1\\\tau&\text{if}~j=2,3,\ldots ,q\\\phi&\text{if}~j=q+1\end{cases}.
\end{eqnarray*}
This gives the following exact sequence of $\lceil\frac{3n}{2}-3\rceil(E^{3q-1})$-modules:
\begin{eqnarray}\label{eq.t5}\begin{tikzpicture}
  \matrix (m) [matrix of math nodes,row sep=3em,column sep=2em,minimum width=1em]
  {
     0 & N^{3q-1} & \left \lceil \frac{3n}{2}-3\right \rceil(L_{1}) & \cdots & \left\lceil \frac{3n}{2}-3\right \rceil(L_{q})  & \left\lceil \frac{3n}{2}-3\right \rceil(L_{q+1}) & 0, \\};
  \path[-stealth]
    (m-1-2) edge node [left] {$ $} (m-1-1)
    (m-1-3) edge node [above] {$\zeta$} (m-1-2)
    (m-1-4) edge node [above] {$\tau$} (m-1-3)
    (m-1-5) edge node [above] {$\tau$} (m-1-4)
    (m-1-6) edge node [above] {$\phi$} (m-1-5)
    (m-1-7) edge node [above] {$ $} (m-1-6); 
\end{tikzpicture}\end{eqnarray}
For $j=1,2,...,q-1$, Lemma \ref{l1}(j) and Proposition \ref{p2.4.2}(b) yields
\begin{eqnarray*}
\left\lceil \frac{3n}{2}-3\right \rceil(L_{j})&=&\left \lceil \frac{3n}{2}-3\right \rceil\left(P^{3q-1}_{n-1+3(j-1)(\frac{n}{2}-1)} \oplus P_{n-1+3(j-1)(\frac{3n}{2}-1)+(\frac{n}{2}-1)}^{3q-1}\right)\\&=&\left \lceil \frac{3n}{2}-3\right \rceil (P^{3q-1}_{n-1+3(j-1)(\frac{n}{2}-1)})\oplus \left \lceil \frac{3n}{2}-3\right \rceil (P_{n-1+3(j-1)(\frac{3n}{2}-1)+(\frac{n}{2}-1)}^{3q-1})\\&=&P^{3q+2}_{n-1+3j(\frac{n}{2}-1)}\oplus P^{3q+2}_{n-1+3j(\frac{n}{2}-1)+(\frac{n}{2}-1)}.
\end{eqnarray*}
A similar computation shows that 
\begin{eqnarray*}
\left\lceil \frac{3n}{2}-3\right \rceil(L_{q})=P^{3q+2}_{(n-1)+3q(\frac{n}{2}-1)}\oplus P^{3q+2}_{n-1+3q(\frac{n}{2}-1)+\frac{n}{2}}~\text{and}~\left\lceil \frac{3n}{2}-3\right \rceil(L_{q+1})=P_{l_{3q+2}}^{3q+2}.
\end{eqnarray*}
Since $N^{3q-1}$ is an $E^{3q+2}$-module and $\lceil\frac{3n}{2}-3\rceil(L_{j})$ for $1\leq j\leq q+1$ are projective $E^{3q+2}$-modules, splicing the exact sequences (\ref{eq.t2}) and (\ref{eq.t5}) yields the following $E^{3q+2}$-projective resolution of $S_{1}^{3q+2}$:
\begin{eqnarray*}\label{eq.t6}\begin{tikzpicture}
  \matrix (m) [matrix of math nodes,row sep=2em,column sep=2em,minimum width=1em]
  {
     0 & S_{1}^{3q+2}& P_{1}^{3q+2} & P_{n-1}^{3q+2}\oplus P_{\frac{3n}{2}-2}^{3q+2} & L_{1}\left \lceil \frac{3n}{2}-3\right \rceil & \cdots & \\ & & &  0 & L_{q+1}\left \lceil\frac{3n}{2}-3\right \rceil & L_{q}\left \lceil \frac{3n}{2}-3\right \rceil \\};
  \path[-stealth]
    (m-1-2) edge node [left] {$ $} (m-1-1)
    (m-1-3) edge node [above] {$\pi_{1}^{3q+2}$} (m-1-2)
    (m-1-4) edge node [above] {$\zeta$} (m-1-3)
    (m-1-5) edge node [above] {$\tau=\mu\zeta$} (m-1-4)
    (m-1-6) edge node [above] {$\tau$} (m-1-5)
    (m-2-6) edge node [left] {$\tau$} (m-1-6)
    (m-2-5) edge node [above] {$\phi$} (m-2-6)
    (m-2-4) edge node [above] {$ $} (m-2-5); 
\end{tikzpicture}\end{eqnarray*}
In particular,
\begin{eqnarray}\label{eq.t7}\begin{tikzpicture}
  \matrix (m) [matrix of math nodes,row sep=2em,column sep=2em,minimum width=2em]
  {
      0 & S_{1}^{3q+2} & W_{0} & W_{1} & W_{2} & \cdots & W_{q+1} & W_{q+2} & 0 \\};
  \path[-stealth]
    (m-1-2) edge node [left] {$ $} (m-1-1)
    (m-1-3) edge node [above] {$d_{0}$} (m-1-2)
    (m-1-4) edge node [above] {$d_{1}$} (m-1-3)
    (m-1-5) edge node [above] {$d_{2}$} (m-1-4)
    (m-1-6) edge node [above] {$d_{3}$} (m-1-5)
    (m-1-7) edge node [above] {$d_{q+1}$} (m-1-6)
    (m-1-8) edge node [above] {$d_{q+2}$} (m-1-7)
    (m-1-9) edge node [above] {$ $} (m-1-8);
\end{tikzpicture}\end{eqnarray}
is an $E^{3q+2}$-projective resolution of $S_{1}^{3q+2}$, where
\begin{eqnarray*}
W_{j}&=&\begin{cases}P_{1}^{3q+2}&\text{if}~j=0\\\\P_{n-1}^{3q+2}\oplus P_{\frac{3n}{2}-2}^{3q+2}&\text{if}~j=1\\\\\lceil\frac{3n}{2}-3\rceil(L_{j-1})&~\text{if}~j=2,...,q+2\end{cases}\\&=&\begin{cases}P_{1}^{3q+2}&~\text{if}~j=0\\\\P^{3q+2}_{(n-1)+3(j-1)(\frac{n}{2}-1)}\oplus P^{3q+2}_{(n-1)+3(j-1)(\frac{n}{2}-1)+(\frac{n}{2}-1)}&~\text{if}~j=1,2,\ldots ,q\\\\P_{(n-1)+3q(\frac{n}{2}-1)}^{3q+2}\oplus P^{3q+2}_{(n-1)+3q(\frac{n}{2}-1)+\frac{n}{2}}&~\text{if}~j=q+1\\\\P^{3q+2}_{l_{3q+2}}&~\text{if}~j=q+2\end{cases}
\end{eqnarray*}
and
\begin{eqnarray*}
d_{j}=\begin{cases}\pi_{1}^{3q+2}&\text{if}~j=0\\\zeta&\text{if}~j=1\\\tau&\text{if}~j=2\\g_{j-1}&\text{if}~j=3,\ldots ,q+2\end{cases}=\begin{cases}\pi_{1}^{3q+2}&\text{if}~j=0\\\zeta &\text{if}~j=1\\\tau &\text{if}~j=2,3,\ldots ,q+1\\\phi &\text{if}~j=q+2\end{cases}
\end{eqnarray*}
By Theorem \ref{t1.1.1}, $0\leftarrow S_{1}^{3q+2}\stackrel{d_{0}}\longleftarrow P_{1}^{3q+2}$ is a projective cover for $S_{1}^{3q+2}$. Moreover, $\Im (d_{1})=\ker d_{0}=J(W_{0})=J(P_{1}^{3q+2})$, and a quick calculation shows that $\Im(d_{2})\subseteq J(W_{1})$. Minimality of (\ref{eq.t3}) implies that \begin{eqnarray*}
\Im \left( L_{j}\stackrel{f_{j}}\longrightarrow L_{j-1}\right) \subseteq J(L_{j-1})~\text{for}~1\leq j\leq q+1.
\end{eqnarray*}
Since $P_{1}$ is not a direct summand of $L_{i}$ for $1\leq i\leq q+1$, Proposition \ref{exactness}(c) yields
\begin{eqnarray*}
\Im (d_{j+1})=\Im (g_{j})=\Im\left(\left\lceil \dfrac{3n}{2}-3\right\rceil(f_{j})\right)\subseteq J\left(\left \lceil \frac{3n}{2}-3\right \rceil(L_{j-1})\right)=J(W_{j})\text{ for }2\leq j\leq q+1.
\end{eqnarray*}
Hence, (\ref{eq.t7}) is a minimal projective resolution for $S_{1}^{3q+2}$, as desired. The second part is a consequence of what we just proved.\end{proof}

The following theorem covers the cases when $i$ is congruent to zero or 1 mod 3 (proofs are similar to the one given in Theorem \ref{t2.4.1}).

\begin{thm}\label{t2.4.3} (a) If $q\geq 1$, then 
\begin{eqnarray*}
0\leftarrow S_{1}^{3q}\stackrel{d_{0}}\longleftarrow W_{0}\stackrel{d_{1}}\longleftarrow W_{1}\stackrel{d_{2}}\longleftarrow W_{2}\stackrel{d_{3}}\longleftarrow \cdots \stackrel{d_{q}}\longleftarrow W_{q}\stackrel{d_{q+1}}\longleftarrow W_{q+1}\leftarrow 0
\end{eqnarray*}
is a minimal projective resolution for $S_{1}^{3q}$, where
\begin{eqnarray*}
W_{j}&=&\begin{cases}P_{1}^{3q}&~\text{if}~j=0\\\\P^{3q}_{(n-1)+3(j-1)(\frac{n}{2}-1)}\oplus P^{3q}_{(n-1)+3(j-1)(\frac{n}{2}-1)+(\frac{n}{2}-1)}&~\text{if}~j=1,2,\ldots ,q\\\\P^{3q}_{l_{3q}}&~\text{if}~j=q+1\end{cases}\\d_{j}&=&\begin{cases}\pi_{1}^{3q}&~\text{if}~j=0\\\zeta &~\text{if}~j=1\\\tau &~\text{if}~j=2,\ldots ,q\\\eta &~\text{if}~j=q+1\end{cases}
\end{eqnarray*}
In particular, $\Pd_{E^{3q}}(S_{1}^{3q})=q+1$ for $q\in \mathbb{N}$. Therefore, $q+1\leq\gldim(E^{3q})\leq l_{3q}$ for $q\in \mathbb{N}$.

(b) If $q\geq 1$, then 
\begin{eqnarray*}
0\leftarrow S_{1}^{3q+1}\stackrel{d_{0}}\longleftarrow W_{0}\stackrel{d_{1}}\longleftarrow W_{1}\stackrel{d_{2}}\longleftarrow W_{2}\stackrel{d_{3}}\longleftarrow \cdots \stackrel{d_{q}}\longleftarrow W_{q}\stackrel{d_{q+1}}\longleftarrow W_{q+1}\leftarrow 0
\end{eqnarray*}
is a minimal projective resolution for $S_{1}^{3q+1}$, where
\begin{eqnarray*}
W_{j}&=&\begin{cases}P_{1}^{3q+1}&~\text{if}~j=0\\\\P^{3q+1}_{(n-1)+3(j-1)(\frac{n}{2}-1)}\oplus P^{3q+1}_{(n-1)+3(j-1)(\frac{n}{2}-1)+(\frac{n}{2}-1)}&~\text{if}~j=1,2,\ldots ,q\\\\P^{3q+1}_{l_{3q+1-(\varepsilon -1)}}&~\text{if}~j=q+1\end{cases}\\d_{j}&=&\begin{cases}\pi_{1}^{3q+1}&~\text{if}~j=0\\\zeta &~\text{if}~j=1\\\tau &~\text{if}~j=2,\ldots ,q\\\sigma &~\text{if}~j=q+1\end{cases}
\end{eqnarray*}
In particular, $\Pd_{E^{3q+1}}(S_{1}^{3q+1})=q+1$ for $q\in \mathbb{N}_{0}$. Therefore, $q+1\leq\gldim(E^{3q+1})\leq l_{3q+1}$ for $q\in \mathbb{N}_{0}$.\end{thm}

\subsection{Constructing Endomorphism Rings of Global Dimension Two.}\label{third main theorem}

Throughout this section, we assume the radical chain (\ref{radical chain of family of starting rings}), the module $M^{i}$, and the ring $E^{i}$ are constructed via the greedy construction for each $i\in \mathbb{N}$. Observe that $R^{1}_{1}=R(\mathcal{H}(1))$, where $\mathcal{H}(1)=\langle A(1)\rangle$ and $\displaystyle A(1)=\left\lbrace 0,n,\dfrac{3n}{2}+w:w=0,1,\ldots,n-1\right\rbrace$. Moreover, $R^{1}_{2}=R\left(\left\langle\left\lbrace 0,\dfrac{n}{2}+w:w=0,1,...,\dfrac{n}{2}-1\right\rbrace\right\rangle\right)$ and $R^{1}_{3}=k[[t]]$. For $i\geq 2$, $R^{i}_{l_{i}}=k[[t]]$ and $R^{i}_{1}=R(\mathcal{H}(i))$, where $\mathcal{H}(i)=\langle A(i)\rangle$ and 
\begin{eqnarray*}
A(i)=\displaystyle\left\lbrace  0,\frac{bn}{2},a+1+(i-2)\dfrac{n}{2}+w:b=2,3,\ldots ,i+1,~w=0,1,\ldots,n-1\right\rbrace .
\end{eqnarray*}
Moreover, $R^{i}_{l_{i}-1}=\displaystyle R\left(\left\langle\left\lbrace 0,a+1-\frac{3n}{2}+w:w=0,1,\ldots,\dfrac{n}{2}-1\right\rangle\right\rbrace\right)$, and for $2\leq j\leq l_{i}-2$
\begin{eqnarray*}
R^{i}_{j}=R\left(\left\langle\left\lbrace \frac{bn}{2},a+1+(i-j-2)\dfrac{n}{2}+w:b=0,1,2,\ldots ,i-j+1,w=0,1,\ldots ,\dfrac{n}{2}-1\right\rbrace\right\rangle\right).
\end{eqnarray*} 
\hspace{0.5cm} The second main result of this paper is that $\gldim(E^{i})=2$ for all $i\in \mathbb{N}$ (Theorem \ref{t3.3.1}). Firstly, an immediate consequence of this construction is the following results which we state as a proposition for future reference.
\begin{proposition}\label{p3.3} Fix $i\in \mathbb{N}$. For a radical chain (\ref{radical chain of family of starting rings}), the following holds.
\\(a) For each $i\in\mathbb{N}$ we have $l_{i}=i+2$. Consequently, $l_{i+1}=l_{i}+1$.
\\(b) Identify $P^{i}_{j}$ with its non-zero row and let $(P^{i}_{j})_{b}$ be the $b$-th entry in $P^{i}_{j}$. Then, 
\begin{eqnarray*}
(P_{j}^{i})_{b}=(E^{i})_{jb}=\begin{cases}R^{i}_{j,0}&\text{if}~1\leq b\leq j\\R^{i}_{j,b-j}&\text{if}~j+1\leq b\leq l_{i}\end{cases}.
\end{eqnarray*}
(c) For each integer $i\geq 1$, $1\leq j\leq l_{i}$, and $1\leq b\leq i-j+2$, we have $R^{i}_{j,b}=t^{e(R^{i}_{j})}R^{i}_{j+1,b-1}$.
\\(d) For $i\geq 3$ and $3\leq j\leq l_{i}$ we have $R^{i}_{j}=R^{i-1}_{j-1}$. In particular, $\lceil 1\rceil (P_{j-1}^{i-1})=P_{j}^{i},~\lceil 1\rceil (J(P^{i-1}_{j-1}))=J(P^{i}_{j}),~\text{and}~\lceil 1\rceil (S_{j-1}^{i-1})\cong S^{i}_{j}$ (as $E^{i}$-modules).
\end{proposition}
\begin{example} \normalfont Setting $n=6,~a=10$, and $i=2$ gives $R_{1}^{2}=k[[t^{6},t^{9},t^{11},t^{13},t^{14},t^{16}]]$. Then $R_{2}^{2}=\End_{R_{1}}(m_{1})=k[[t^{3},t^{5},t^{7}]]$, $R_{3}^{2}=k[[t^{2},t^{3}]]$, and $R_{4}^{2}=\End_{R_{1}}(m_{2})=k[[t]]$. Moreover,
\begin{eqnarray*}
E=\left(\begin{matrix}R^{2}_{1,0}&R^{2}_{1,1}&R^{2}_{1,2}&R^{2}_{1,3}\\R^{2}_{2,0}&R^{2}_{2,0}&R^{2}_{2,1}&R^{2}_{2,2}\\R^{2}_{3,0}&R^{2}_{3,0}&R^{2}_{3,0}&R^{2}_{3,1}\\R^{2}_{4,0}&R^{2}_{4,0}&R^{2}_{4,0}&R^{2}_{4,0}\end{matrix}\right),
\end{eqnarray*}
and
\begin{eqnarray*}
R^{2}_{1,1}&=&t^{6}R^{2}_{2,0},~R^{2}_{1,2}=t^{6}R^{2}_{2,1},~R^{2}_{1,3}=t^{6}R^{2}_{2,2},\\R^{2}_{2,1}&=&t^{3}R^{2}_{3,0},~R^{2}_{2,2}=t^{3}R^{2}_{3,1},\\R^{2}_{3,1}&=&t^{2}R^{2}_{4,0}.
\end{eqnarray*}\end{example}
\hspace{0.5cm} Unless otherwise stated in the calculation below we are using the identification used in Example \ref{id}. We now prove that the projective dimension of the first simple module is always one.
\begin{lemma}\label{f.simple} For each $i\in \mathbb{N}$, the minimal projective resolution of $S_{1}^{i}$ is 
\begin{eqnarray*}
0\longleftarrow S_{1}^{i}\stackrel{\pi_{1}^{i}}\longleftarrow P_{1}^{i}\stackrel{t^{n}}\longleftarrow P_{2}^{i}\longleftarrow 0.
\end{eqnarray*} 
In particular, $\Pd_{E^{i}}(S_{1}^{i})=1$ for all $i\in \mathbb{N}$.\end{lemma}
\begin{proof} Notice that $e(R^{i}_{1})=n$ for all $i\in \mathbb{N}$. By Lemma \ref{p3.3} (b), $(P^{i}_{1})_{b}=R^{i}_{1,b-1}~\text{for}~1\leq b\leq l_{i}$. It follows that
\begin{eqnarray*}
(\ker (\pi_{1}^{i}))_{b}=(J(P^{i}_{1}))_{b}=\begin{cases}R^{i}_{1,1}&\text{if}~b=1\\R^{i}_{1,b-1}&\text{if}~2\leq b\leq l_{i}\end{cases}=\begin{cases}t^{n}R^{i}_{2,0}&\text{if}~b=1\\t^{n}R^{i}_{2,b-2}&\text{if}~2\leq b\leq l_{i}\end{cases}=t^{n}P_{2}^{i}.
\end{eqnarray*}
where $\ker(\pi_{1}^{i})=J(P^{i}_{1})$ is identified with its first row and $(\ker (\pi_{1}^{i}))_{b}=(J(P^{i}_{1}))_{b}$ is the $b$-th entry.\end{proof}

A similar calculation to the one given in Lemma \ref{f.simple} proves the following lemma. 

\begin{lemma}\label{resolution} (a) $\gldim(E^{1})=\gldim(E^{2})=2$.
\\(b) The minimal projective resolution of $S_{2}^{i},~S_{3}^{i},~S_{l_{i}-1}^{i},$ and $S_{l_{i}}^{i}$ are as follows:
\begin{eqnarray*}
\begin{tikzpicture}
  \matrix (m) [matrix of math nodes,row sep=2em,column sep=2em,minimum width=2em]
  {
    for~all~i\in \mathbb{N},~  0 & S_{2}^{i} & P_{2}^{i} & P_{1}^{i}\oplus P_{3}^{i} & P_{2}^{i} & 0,\\};
  \path[-stealth]
    (m-1-2) edge node [left] {$ $} (m-1-1)
    (m-1-3) edge node [above] {$\pi_{2}^{i}$} (m-1-2)
    (m-1-4) edge node [above] {$(1~t^{\frac{n}{2}})$} (m-1-3)
    (m-1-5) edge node [above] {$\left(\begin{matrix}t^{n}\\-t^{\frac{n}{2}}\end{matrix}\right)$} (m-1-4)
    (m-1-6) edge node [left] {$ $} (m-1-5);
\end{tikzpicture}
\\
\begin{tikzpicture}
  \matrix (m) [matrix of math nodes,row sep=2em,column sep=2em,minimum width=2em]
  {
     \text{for}~i\geq 3,~ 0 & S_{3}^{i} & P_{3}^{i} & P_{2}^{i}\oplus  P_{4}^{i} & P_{3}^{i} & 0,\\};
  \path[-stealth]
    (m-1-2) edge node [left] {$ $} (m-1-1)
    (m-1-3) edge node [above] {$\pi_{3}^{i}$} (m-1-2)
    (m-1-4) edge node [above] {$(1~t^{\frac{n}{2}})$} (m-1-3)
    (m-1-5) edge node [above] {$\left(\begin{matrix}t^{\frac{n}{2}}\\-1\end{matrix}\right)$} (m-1-4)
    (m-1-6) edge node [left] {$ $} (m-1-5);
\end{tikzpicture}
\\\begin{tikzpicture}
  \matrix (m) [matrix of math nodes,row sep=2em,column sep=2em,minimum width=2em]
  {
    \text{for}~i\geq 2,~  0 & S_{l_{i}-1}^{i} & P_{l_{i}-1}^{i} & P_{l_{i}-2}^{i}\oplus P_{l_{i}}^{i} & P_{l_{i}-1}^{i} & 0,\\};
  \path[-stealth]
    (m-1-2) edge node [left] {$ $} (m-1-1)
    (m-1-3) edge node [above] {$\pi_{l_{i}-1}^{i}$} (m-1-2)
    (m-1-4) edge node [above] {$(1~t^{\varepsilon})$} (m-1-3)
    (m-1-5) edge node [above] {$\left(\begin{matrix}t^{\frac{n}{2}}\\-t^{\frac{n}{2}-\varepsilon}\end{matrix}\right)$} (m-1-4)
    (m-1-6) edge node [left] {$ $} (m-1-5);
\end{tikzpicture}
\\\begin{tikzpicture}
  \matrix (m) [matrix of math nodes,row sep=2em,column sep=2em,minimum width=2em]
  {
     \text{for}~i\geq 2,~ 0 & S_{l_{i}}^{i} & P_{l_{i}}^{i} & P_{l_{i}-1}^{i}\oplus P_{l_{i}}^{i} & P_{l_{i}}^{i} & 0,\\};
  \path[-stealth]
    (m-1-2) edge node [left] {$ $} (m-1-1)
    (m-1-3) edge node [above] {$\pi_{l_{i}}^{i}$} (m-1-2)
    (m-1-4) edge node [above] {$\left(\begin{matrix}1,t\end{matrix}\right)$} (m-1-3)
    (m-1-5) edge node [above] {$\left(\begin{matrix}t^{\varepsilon}\\-t^{\varepsilon -1}\end{matrix}\right)$} (m-1-4)
    (m-1-6) edge node [left] {$ $} (m-1-5);
\end{tikzpicture}
\end{eqnarray*}\end{lemma} 
Now we prove the second main result of this paper.
\begin{thm}\label{t3.3.1} (a) For $i\geq 3$ and $3\leq j\leq l_{i}-2$, the minimal projective resolutions of the simple $S_{j}^{i}$ is:
\begin{eqnarray*}\label{eq.7}
\begin{tikzpicture}
  \matrix (m) [matrix of math nodes,row sep=2em,column sep=2em,minimum width=2em]
  {
      0 & S_{j}^{i} & P_{j}^{i} & P_{j-1}^{i}\oplus P_{j+1}^{i} & P_{j}^{i} & 0, \\};
  \path[-stealth]
    (m-1-2) edge node [left] {$ $} (m-1-1)
    (m-1-3) edge node [above] {$\pi_{j}^{i}$} (m-1-2)
    (m-1-4) edge node [above] {$(1~t^{\frac{n}{2}})$} (m-1-3)
    (m-1-5) edge node [above] {$\left(\begin{matrix}t^{\frac{n}{2}}\\-1\end{matrix}\right)$} (m-1-4)
    (m-1-6) edge node [left] {$ $} (m-1-5);
\end{tikzpicture}
\end{eqnarray*}
\noindent (b) $\gldim(E^{i})=2$ for all $i\in \mathbb{N}$.\end{thm}

\begin{proof} (a) We proceed by induction on $i$. For $i=3,~l_{i}=5$, and Lemma \ref{resolution} gives the desired result for $S^{3}_{3}$. Assume the result is true for $i-1$ (where $i-1\geq 3$). The minimal projective resolution of $S_{3}^{i}$ is given by Lemma \ref{resolution}(b). If $4\leq j\leq l_{i}-2$, then $3\leq j-1\leq l_{i}-3=l_{i-1}-2$, and the induction hypothesis gives the following minimal resolution of $S^{i-1}_{j-1}$;
\begin{eqnarray}\label{1}
0\leftarrow S_{j-1}^{i-1}\stackrel{\pi_{j-1}^{i-1}}\longleftarrow P_{j-1}^{i-1}\stackrel{(1~t^{\frac{n}{2}})}\longleftarrow P_{j-2}^{i-1}\oplus P_{j}^{i-1}\stackrel{\left(\begin{matrix}t^{\frac{n}{2}}\\-1\end{matrix}\right)}\longleftarrow P_{j-1}^{i-1}\leftarrow 0.
\end{eqnarray}
Since all indices appearing in the minimal resolution of $S^{i-1}_{j-1}$ are greater than one, applying $\lceil 1\rceil$ to (\ref{1}) and using Proposition \ref{exactness} (c) and (d) gives the following $E^{i-1}$-minimal projective resolution for $\lceil 1\rceil(S^{i-1}_{j-1})$:
\begin{eqnarray}\label{2}
\begin{tikzpicture}
  \matrix (m) [matrix of math nodes,row sep=2em,column sep=3em,minimum width=2em]
  {
      0 & \lceil 1\rceil(S_{j-1}^{i-1}) & \lceil 1\rceil(P_{j-1}^{i-1}) & \lceil 1\rceil(P_{j-2}^{i-1}\oplus P_{j}^{i-1}) & \lceil 1\rceil(P_{j-1}^{i-1}) & 0. \\};
  \path[-stealth]
    (m-1-2) edge node [left] {$ $} (m-1-1)
    (m-1-3) edge node [above] {$\lceil 1\rceil(\pi_{j-1}^{i-1})$} (m-1-2)
    (m-1-4) edge node [above] {$(1~t^{\frac{n}{2}})$} (m-1-3)
    (m-1-5) edge node [above] {$\left(\begin{matrix}t^{\frac{n}{2}}\\-1\end{matrix}\right)$} (m-1-4)
    (m-1-6) edge node [left] {$ $} (m-1-5);
\end{tikzpicture}
\end{eqnarray}
Proposition \ref{l1} (j) and Proposition \ref{p3.3}(d) gives the following minimal $E^{i}$-projective resolution of $S^{i}_{j}$;
\begin{eqnarray*}
\begin{tikzpicture}
  \matrix (m) [matrix of math nodes,row sep=2em,column sep=2em,minimum width=2em]
  {
      0 & S_{j}^{i} & P_{j}^{i} & P_{j-1}^{i}\oplus P_{j+1}^{i} & P_{j}^{i} & 0. \\};
  \path[-stealth]
    (m-1-2) edge node [left] {$ $} (m-1-1)
    (m-1-3) edge node [above] {$\pi_{j}^{i}$} (m-1-2)
    (m-1-4) edge node [above] {$(1~t^{\frac{n}{2}})$} (m-1-3)
    (m-1-5) edge node [above] {$\left(\begin{matrix}t^{\frac{n}{2}}\\-1\end{matrix}\right)$} (m-1-4)
    (m-1-6) edge node [left] {$ $} (m-1-5);
\end{tikzpicture}
\end{eqnarray*}
\noindent (b) This is a direct consequence of part (a) and Lemmas \ref{f.simple} and \ref{resolution}. \end{proof}

\section{Acknowledgements}

This paper is based on the authors thesis at the University of Toronto. The author dedicates this paper to the memory of Ragnar-Olaf Buchweitz for being an incredible supervisor who showed tremendous patience, expertise, and for all the discussions and feedbacks. More thanks go to Marco Gualtieri and Joe Repka for very helpful discussions, Osamu Iyama for his feedback on the thesis, and Graham J. Leuschke for his feedback on the paper.



\begin{thebibliography}{99}

\bibitem{A2} M. Auslander, $A~Functorial~Approach~to~Representation~theory$, Representation of algebras (Puebla, 1980), Lecture Notes, 944, Berlin, New York: Springer-Verlag, pp. 105-179.

\bibitem{A3} M. Auslander, $Rational~singularities~and~almost~split~sequences$. Trans. Amer. Math. Soc. 293 (1986), no.2, 511-531.

\bibitem{A1} M. Auslander, $Representation~dimension~of~Artin~algebras$. Unpublished Lecture notes, Queen Mary College, London, 1971.

\bibitem{ARS} M. Auslander, I. Reiten, S. O. Smalo: $Representation~theory~of~Artin~algebras$. Cambridge studies of Advanced Mathematics, 36. Cambridge University Press, Cambridge, 1995.

\bibitem{Bass}
Hyman Bass, $Algebraic~K$-$theory$. W. A. Benjamin, Inc., New York-Amsterdam, 1968. 

\bibitem{Grauert}
H. Grauert and R. Remmert, $Analytische~Stellenalgebren$. (German) Unter Mitarbeit von. Riemenschneider. Die Grundlehren der mathematischen Wissenschaften, Band 176. Springer-Verlag, Berlin-New York, 1971.  (47:5290)

\bibitem{Hong}
Jooyoun Hong, Bernd Ulrich, and V. Wolmer Vasconcelos, $Normalization~of~modules$. J. Algebra 303 (2006), no. 1, 133-145.

\bibitem{Krause}
Henning Krause, $Krull-Schmidt~categories~and~projective~covers$. Expo. Math. 33 (2015), no. 535-549. 

\bibitem{OI}
Osamu Iyama, $\tau$-$categories.~II.~Nakayama~pairs~and~rejective~subcategories$. Algebr. Represent. Theory 8 (2005), no. 4, 449-477. 

\bibitem{OI2}
Osamu Iyama, $Rejective~subcategories~of~artin~algebras~and~orders$. arXiv 0311281v1, 2003.

\bibitem{GL}
Graham J. Leuschke, $Endomorphism~rings~of~finite~global~dimension$. Canad. J.~Math, 59, no.2, (2007), 332-343.

\bibitem{HM}
Hideyuki Matsumura, $Commutative~Ring~Theory$. Translated from the Japanese by M. Reid. Second edition. Cambridge Studies in Advanced Mathematics, vol. 8. Cambridge University Press, Cambridge, 1989. 

\bibitem{MR}
J. C. McConnell and J. C. Robson, $Noncommutative~Noetherian~rings$. With the cooperation of L. W. Small. Revised edition. Graduate studies in Mathematics, 30. American Mathematical Society, Providence, RI, 2001. MR1811901 (2001i:16039) 

\bibitem{Neeman}
Amnon Neeman, $Triangulated~Categories$. Annals of Mathematics Studies, 148. Princeton University Press, Princeton, NJ, 2001.

\bibitem{IR}
Irving Reiner, $Maximal~orders$. Corrected reprint of the 1975 original. With a foreword by M. J. Taylor. London Mathematical Society Monographs, New Series, 28. The Clarendon Press, Oxford University Press, Oxford, 2003. 

\bibitem{IR2}
Irving Reiner, $Maximal~orders$. London Mathematical Society Monographs, no. 5. Academic Press [A subsidiary of Harcourt Brace Jovanovich,
Publishers], London-New York, 1975.

\bibitem{Rosales1}
J. C. Rosales and P. A. Garcia-Sanchez, $Numerical~semigroups$, Developments in Mathematics, 20. Springer, New York, 2009. 

\bibitem{Rosales2}
J. C. Rosales and P. A. Garcia-Sanchez, $Finitely~generated~commutative~monoids$, Nova Science Publisher, Inc, Commack, New York, 1999. 

\bibitem{WV}
V. Wolmer Vasconcelos, $On~the~radical~of~endomorphism~rings~of~local~modules$. Bull. Braz. Math. Soc. (N.S.) 45 (2014), no.4, 871-886. 


\end{thebibliography}
\end{document}